\documentclass[leqno,10pt]{amsart}

\usepackage[english]{babel}
\usepackage[T1]{fontenc}

\usepackage{amsmath,amssymb,amsthm}
\usepackage{mathtools}  
\usepackage[margin=1in]{geometry}

\usepackage{graphics,tikz,caption,subcaption}

\usepackage{mathrsfs}
\usepackage{stmaryrd}

\usepackage[colorlinks	=	true, linkcolor	=	red, urlcolor	=	red, citecolor	=	red]{hyperref}
\usepackage{bookmark}									
\usepackage{cleveref}

\usepackage{multicol}
\usepackage{todonotes}
\usepackage{enumitem}
\usepackage{verbatim}

\theoremstyle{definition}
\newtheorem{thm}{Theorem}[section]
\newtheorem{exm}[thm]{Example}
\newtheorem{defi}[thm]{Definition}
\newtheorem{lemm}[thm]{Lemma}

\newtheorem{rem}[thm]{Remark}

\newtheorem{prop}[thm]{Proposition}
\newtheorem{ques}[thm]{Question}

\theoremstyle{remark}

\DeclareMathOperator{\dist}{dist}

\DeclareMathOperator{\lip}{lip}

\allowdisplaybreaks

\newcommand{\apmd}[2][]{															
	\ifthenelse{\equal{#1}{}}%
					{ \operatorname{N}_{#2}	}%
					{ \operatorname{N}_{#1}(#2) 	}}
					
	\title{Bi-Lipschitz embedding metric triangles in the plane}
    \author{Xinyuan Luo}
    \address{Department of Mathematical Sciences, Stevens Institute of Technology, Hoboken, NJ 07030, USA}
    \email{xluo15@stevens.edu}

    \author{Matthew Romney}
    \address{Department of Mathematics, University of Hawaii at Manoa, Honolulu, Hawaii 96822, USA.}
    \email{mromney@hawaii.edu} 

    \author{Alexandria L. Tao} 
    \address{Department of Mathematical Sciences, Stevens Institute of Technology, Hoboken, NJ 07030, USA.}
    \email{atao@stevens.edu} 

\begin{document}
 	\maketitle

    \begin{abstract}
        A \textit{metric polygon} is a metric space comprised of a finite number of closed intervals joined cyclically. The second-named author and Ntalampekos recently found a method to bi-Lipschitz embed an arbitrary metric triangle in the Euclidean plane with uniformly bounded distortion, which we call here the \textit{tripodal embedding}. In this paper, we prove the sharp distortion bound $4\sqrt{7/3}$ for the tripodal embedding. We also give a detailed analysis of four representative examples of metric triangles: the intrinsic circle, the three-petal rose, tripods and the twisted heart. In particular, our examples show the sharpness of the tripodal embedding distortion bound and give a lower bound for the optimal distortion bound in general. Finally, we show the triangle embedding theorem does not generalize to metric quadrilaterals by giving a family of examples of metric quadrilaterals that are not bi-Lipschitz embeddable in the plane with uniform distortion.
    \end{abstract}
	
    \let\thefootnote\relax

    \section{Introduction}

    The \textit{bi-Lipschitz embedding problem} asks one to give an intrinsic characterization of those metric spaces that can be embedded in some Euclidean space under a bi-Lipschitz map. Moreover, if a metric space does embed in some Euclidean space, one may then attempt to find or estimate the optimal target dimension and bi-Lipschitz distortion. The bi-Lipschitz embedding problem is wide open in full generality and has produced a rich body of research from a variety of mathematical points of view. We refer the reader to \cite{Hei:03}, \cite{LP:01}, \cite[Chapter 15]{Mat:02}, \cite{NN:12}, \cite{Ost:13} and \cite{Sem:99} for general background.

    In this article, we study a specific class of metric spaces, namely, \textit{metric polygons}, and their embeddings in Euclidean spaces, with a focus on \textit{metric triangles}. We say that a metric polygon is a metric space comprised of finitely many closed intervals (i.e., metric spaces each isometric to $[0,a]$ for some $a>0$), called edges, joined in cyclic fashion. For two points lying on different edges, there is no constraint on the distance between them beyond what follows from the triangle inequality. A metric polygon with three edges is called a metric triangle, with four edges is called a metric quadrilateral, and so forth. See \Cref{defi:polygon} below for a formal definition. 

    To our knowledge, metric polygons first appear explicitly in the recent paper \cite{NR:23}, in which it is shown that any metric triangle embeds in $\mathbb{R}^2$ (equipped with the Euclidean metric) under a bi-Lipschitz map with distortion at most $12$.  The fact is then applied to the problem of finding good polyhedral approximations of metric surfaces under minimal geometric assumptions. We refer to the embedding found in \cite{NR:23} as the \textit{tripodal embedding}; it will be reviewed in \Cref{sec:embedding} below. Our goal is to explore the topic of embedding metric polygons more systematically and extend the embedding result from \cite{NR:23} in several directions. 

    As our main objective, we find the sharp bound on the distortion of the tripodal embedding in the general case. 

    \begin{thm} \label{thm:embedding}
        Every metric triangle embeds in $\mathbb{R}^2$ under the tripodal embedding with distortion at most $D = 4\sqrt{7/3} \approx 6.11$.
    \end{thm}
    In \Cref{sec:heart}, we use an example that we call the \textit{twisted heart} to show that the constant $4\sqrt{7/3}$ cannot be improved. To our knowledge, this example has not been considered in the literature previously. This leaves open the question of the best possible value for the distortion $D$ with an arbitrary embedding. Let $\lip(\triangle)$ denote the minimum value for which any metric triangle embeds in $\mathbb{R}^2$ with distortion at most $\lip(\triangle)$. A simple example, called the \textit{three-petal rose}, shows that $\lip(\triangle) \geq 2$; see \Cref{sec:rose}. The example of the twisted heart suggests an even larger value for $\lip(\triangle)$, perhaps roughly the value $4$.  However, we have not been able to find any general embedding method that improves upon the tripodal embedding.
    
    \begin{ques}
        What is the value of $\lip(\triangle)$?
    \end{ques} 

    In \Cref{sec:examples}, we study the previously mentioned examples, along with two other representative examples: tripods and the intrinsic circle. Specifically, we consider the problem of finding distortion-minimizing bi-Lipschitz embeddings for these triangles. We can think of such a distortion-minimizing embedding as an optimal ``Euclidean realization'' of the original triangle. Note, however, that such an embedding is typically not unique; since the distortion of an embedding depends only on the worst pair of points in the space, one can typically perturb a given embedding elsewhere without affecting the distortion.

    In the final main section of this paper, \Cref{sec:quadrilaterals}, we consider the situation for arbitrary metric polygons. Note that it follows immediately from general properties of bi-Lipschitz embeddings (see Theorem 3.2 in \cite{LP:01}) that every metric $n$-gon bi-Lipschitz embeds in $\mathbb{R}^{2n+1}$ with uniformly bounded distortion depending only on $n$. Thus our main question is, for a given $n$, to determine the smallest dimension for which one obtains an embedding of every $n$-gon with uniformly bounded distortion. We show that the analogue of \Cref{thm:embedding} is false for metric quadrilaterals, even with simple topology. 

    \begin{prop} \label{prop:bad_quadrilaterals}
        There is a sequence of metric quadrilaterals $X_n$, each homeomorphic to the circle, such that $\lip(X_n,\mathbb{R}^2)$ is unbounded as $n \to \infty$. 
    \end{prop}

    This proposition is related to the fact that a simple closed curve separates the plane. On the other hand, we expect that there is no obstruction to embedding arbitrary metric $n$-gons in $\mathbb{R}^3$ with distortion bound depending only on $n$.

    \begin{ques}
        For each $n \geq 4$, is there a value $D_n$ such that every metric $n$-gon embeds in $\mathbb{R}^3$ with distortion $D_n$?
    \end{ques}

    We also ask whether every metric quadrilateral is individually bi-Lipschitz embeddable in $\mathbb{R}^2$, i.e., without a uniform bound on the bi-Lipschitz constant. 

    \begin{ques}
        Is every metric quadrilateral bi-Lipschitz embeddable in $\mathbb{R}^2$? 
    \end{ques}
    On the other hand, $K_5$, the complete graph on five vertices, can be realized as a metric pentagon. That is, a metric pentagon exists that is homeomorphic to $K_5$.  See \Cref{exm:pentagon} below. We recall Kuratowsk's Theorem in graph theory, which states that a graph can be topologically embedded in the plane if and only if it does not contain a subgraph that is a subdivision of $K_5$ or $K_{3,3}$ (the complete bipartite graph on two sets of three vertices). Thus a metric pentagon need not be even topologically embeddable in the plane.

    In addition to our work, there have been other investigations into the metric geometry of classes of one-dimensional metric spaces. For example, low-distortion embeddings of metric graphs in Euclidean space were studied by Linial--London--Rabinovich in \cite{LLR:95}. More recently, bi-Lipschitz embeddings of so-called \textit{quasiconformal trees} have also been studied by David--Eriksson-Bique--Vellis in  \cite{DV:22} and \cite{DEV:23}.

    \section{Background}

    \subsection{Metric spaces and bi-Lipschitz maps}

    We review some of the definitions and concepts related to metric spaces. For a detailed account of the subject, see the monographs of Bridson--Haefliger \cite{BH:99} and Burago--Burago--Ivanov \cite{BBI:01}.  

    \begin{defi}
        Given a set $X$, a \textit{metric} (or \textit{distance function}) on $X$ is a function $d \colon X \times X \to [0,\infty)$ satisfying the following properties for all $x,y,z \in X$:
        \begin{enumerate} \itemsep0em 
            \item $d(x,y)=0$ if and only if $x=y$ (positive definiteness)
            \item $d(x,y) = d(y,x)$ (symmetry)
            \item $d(x,y) \leq d(x,z) + d(z,y)$ (triangle inequality)
        \end{enumerate}
        A \textit{metric space} is a pair $(X,d)$, where $X$ is a set and $d$ is a metric on $X$. 
    \end{defi}
    Perhaps the most familiar example of a metric space is $n$-dimensional Euclidean space, which is the set $\mathbb{R}^n$ equipped with the Euclidean metric $d_{\text{Euc}}(x,y) = \|x - y\| = \sqrt{(x_1-y_1)^2 + \cdots + (x_n - y_n)^2}$. For an arbitrary metric space $(X,d)$, a natural question is how different the metric $d$ is from the Euclidean metric, quantitatively. To make this precise, we use the following definition.
    \begin{defi}
        For each $L>0$, a map $f \colon X \to Y$ between metric spaces $(X,d_X)$ and $(Y,d_Y)$ is \textit{$L$-Lipschitz} if $d_Y(f(x),f(y)) \leq Ld_X(x,y)$ for all $x,y \in X$. The map $f$ is \textit{$L$-bi-Lipschitz} if 
        \[L^{-1}d_X(x,y) \leq d_Y(f(x),f(y)) \leq Ld_X(x,y)\]
        for all $x,y \in X$. The map $f$ is \textit{Lipschitz} (resp. \textit{bi-Lipschitz}) if it is $L$-Lipschitz (resp. $L$-bi-Lipschitz) for some $L>0$. 
    \end{defi}

        The \textit{bi-Lipschitz embedding problem} then asks one to characterize when a given metric space can be mapped onto a subset of Euclidean space of some dimension under a bi-Lipschitz map. As discussed in the introduction, this is a difficult open problem in full generality. To refine this problem, we give the following definition. 

        \begin{defi}
            Let $f\colon X \to Y$ be a bi-Lipschitz map between metric spaces $(X,d_X)$ and $(Y,d_Y)$. The \textit{distortion} of $f$ is $\lip(f) = L_0\cdot L_1$, where $L_0, L_1$ are the minimum values such that $L_0^{-1}d_X(x,y) \leq d_Y(f(x),f(y)) \leq L_1d_X(x,y)$.

            For metric spaces $(X,d_X)$ and $(Y,d_Y)$, we define 
            \[\lip(X,Y) = \inf \lip(f),\]
            where the infimum is taken over all bi-Lipschitz embeddings $f\colon X \to Y$. Note that $\lip(X,Y) = \infty$ if no such embedding exists. 
        \end{defi}

        A \textit{curve} is a continuous function from an interval into a metric space. Let $(X,d)$ denote this target metric space. A curve $\Gamma \colon [a,b] \to X$ is  \textit{closed} if $\Gamma(a) = \Gamma(b)$ and \textit{simple} if it is injective (that is, it has no self-intersections) except that possibly $\Gamma(a) = \Gamma(b)$. 
        
        Given a curve $\Gamma \colon [a,b] \to X$, the \textit{length} of $\Gamma$ is 
        \[\ell(\Gamma) = \sup \sum_{j=1}^n d(\Gamma(t_{j-1}), \Gamma(t_j)),\] 
        where the supremum is taken over all finite partitions $a = t_0 < t_1 < \cdots < t_n = b$.
        A curve $\Gamma$ between two points $x$ and $y$ is a \textit{geodesic} if $\ell(\Gamma) = d(x,y)$. A metric space $(X,d)$ is a \textit{length space} if 
        \[d(x,y)  = \inf( \Gamma)\]
        for all $x,y \in X$, where the infimum is taken over all curves $\Gamma$ whose image contains $x$ and $y$. If $X$ is a compact length space, then by the Hopf--Rinow Theorem (see \cite[Proposition I.3.7]{BH:99}) every two points are connected by at least one geodesic. Given two points $x,y \in X$, we write $[xy]$ to denote some choice of geodesic between the two points. 

        The examples of metric polygons we consider in this paper are all either length spaces or subsets of length spaces. Topologically, they are graphs, so a length metric can be defined simply by specifying the length of each edge. See \cite[Chapter 2]{BBI:01} for a more detailed overview of length spaces. 

        Finally, the \textit{concatenation} of two curves $\Gamma_1 \colon [a,b] \to X$, $\Gamma_2 \colon [c,d] \to X$ satisfying $\Gamma_1(b) = \Gamma_2(c)$ is denoted by $\Gamma_1 * \Gamma_2$. This is the curve from $[a, b + d-c]$ to $X$ defined by the formula $\Gamma_1(t)$ if $t \leq b$ and $\Gamma_2(t - b + c)$ if $t \geq b$.

        \subsection{Metric polygons}

        Our interest in this paper is in a specific class of metric spaces called metric polygons.

    \begin{defi} \label{defi:polygon}
        Let $n \geq 2$. A \textit{metric $n$-gon} is a metric space that is the union of $n$ subspaces $I_1, \ldots, I_n$ that are each isometric to a closed interval $[0,\ell_n]$ and connected at the endpoints cyclically. More precisely, if we let $\varphi_j \colon [0,\ell_j] \to I_j$ denote the required isometry onto the interval $I_j$, then it holds that $\varphi_j(\ell_j) = \varphi_{j+1}(0)$ for all $j \in \{1, \ldots, n\}$, with the identification $n+1=1$. A metric $3$-gon is called a \textit{metric triangle}, and a metric $4$-gon is called a \textit{metric quadrilateral}, and so forth. More generally, a \textit{metric polygon} is a metric $n$-gon for some $n$.

        Each subspace $I_j$ is called an \textit{edge} of the metric polygon. Each endpoint of one of the edges $I_j$ is called a \textit{vertex} of the metric polygon.
    \end{defi}
    
    Note that this definition allows the sets $I_j$ to intersect at points other than the endpoints, although the nature of such additional intersection points is constrained by the triangle inequality. Also, note that the terms ``edge'' and ``vertex'' here are distinct from the topological meaning of these terms for a graph. For example, graphs (such as $K_5$ or $K_{3,3}$, as discussed in \Cref{exm:pentagon}) can be represented as metric polygons. In these cases, the topological edges and vertices of the graph need not coincide with its edges and vertices as a metric polygon.
    
    Next, we introduce some main examples of metric triangles that will be studied further in \Cref{sec:examples}. These are illustrated in \Cref{fig:examples}. In each of these examples, the space is a length space or subset of a length space. 

\begin{exm} \label{exm:examples}
    \begin{itemize}
        \item[(a)] \textbf{The intrinsic circle.} Let $S$ be the unit circle in $\mathbb{R}^2$, which we give the arc length or angular distance. More precisely, each point in $S$ can be represented by some $\theta \in [0,2\pi)$, with the correspondence $\theta \mapsto (\cos(\theta),\sin(\theta))$. Then $d(\theta,\tau) = \min\{|\theta - \tau|,|\theta - \tau - 2\pi|,|\theta - \tau + 2\pi|\}$. 

        To view this space as a metric triangle, we pick any three points that do not all belong to the interior of the same semicircle to be vertices. For example, we can take $\theta_1 = 0$, $\theta_2 = 2\pi/3$ and $\theta_3 = 4\pi/3$.
        \item[(b)] \textbf{The three-petal rose.} Let $S_1, S_2, S_3$ be three copies of the intrinsic circle as just defined. Let $R$ be the wedge sum of $S_1, S_2, S_3$ with basepoint $0$ in each $S_i$ (that is, the quotient space of the disjoint union of $S_1, S_2, S_3$ given by identifying the point $0$ in each $S_i$). We define $d$ to be the length metric on $R$. More precisely, $d$ is defined as follows. If $\theta, \tau \in S_i$ for the same $i$, then $d(\theta, \tau)$ agrees with the metric in (a). If $\theta \in S_i$ and $\tau \in S_j$ for $i \neq j$, then $d(\theta,\tau) = d(\theta,0) + d(0,\tau)$. 

        To view this space as a metric triangle, we take the antipodal point on each circle as vertices. That is, let $p = \pi \in S_1$, $q = \pi \in S_2$ and $r = \pi \in S_3$. The points $0$ and $p$ divide $S_1$ into two semicircles $S_1^1$, $S_1^2$. Likewise, $S_2$ and $S_3$ are divided into semicircles $S_2^1, S_2^2$ and $S_3^1, S_3^2$, respectively. Then the edges of $R$ are $[pq] = S_1^2 \cup S_2^1$, $[qr] = S_2^2 \cup S_3^1$ and $[rp] = S_3^2 \cup S_1^1$. 
        \item[(c)]  \textbf{Tripods.} Choose values $l_i \geq 0$, $i = 1,2,3$,  and consider the three intervals $T_i = [0, l_i]$, $i = 1,2,3$. We define $T$ to be the wedge sum of $T_1, T_2, T_3$ with basepoint $0$ in each $T_i$. We define $d$ to be the length metric on $T$. That is, if $x,y \in T_i$ for the same $i$, then $d(x,y) = |x-y|$. If $x \in T_i$ and $y \in T_j$ for $i\neq j$, then $d(x,y) = d(x,0) + d(0,y) = x+y$. Such a metric space is called a \textit{tripod}, or a \textit{metric tripod} for emphasis. We call the point $0$ the \textit{center} of the tripod. 
        
        Let $p = l_1 \in T_1$, $q = l_2 \in T_2$, $r = l_3 \in T_3$. We can view $T$ as a metric triangle by taking the points $p,q,r$ as vertices and $[pq] = T_1 \cup T_2$, $[qr] = T_2 \cup T_3$, $[rp] = T_3 \cup T_1$ as edges.
        \item[(d)] \textbf{The twisted heart.} The example is less obvious but turns out to be the extremal example for the tripodal embedding. We start with the graph consisting of fourteen edges of unit length as shown in \Cref{fig:heart}. The distance between two points is the length of the shortest curve between them. We obtain a triangle by taking $p,q,r$ as vertices and removing the edges from $a$ to $c$ and from $b$ to $e$ (the red edges in \Cref{fig:heart}) without changing the metric on the remaining points. Observe that each pair of vertices is at distance $4$ apart. Moreover, $d(a,b) = 4$ while $d(a,c) = d(b,e) = 1$. We denote this example by $H$. 
    \end{itemize}
\end{exm}

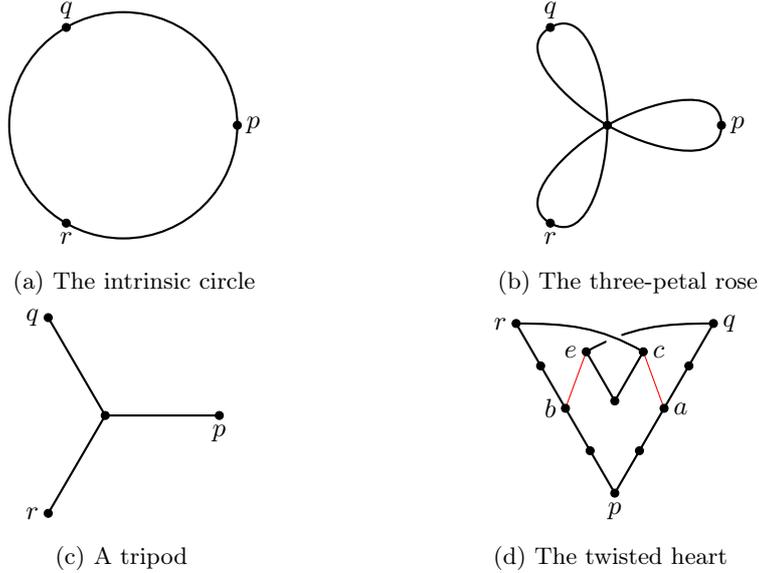
\begin{figure} 
    \centering
    \hfill
        \subfloat[The intrinsic circle]{
             \begin{tikzpicture}[scale=.75]
        \draw[black, thick] (0,0) circle (2);
        \filldraw[black] (2,0) circle (2pt) node[anchor=west]{$p$};
        \filldraw[black] (-1,1.73) circle (2pt) node[anchor=south]{$q$};
        \filldraw[black] (-1,-1.73) circle (2pt) node[anchor=north]{$r$};
        \end{tikzpicture}
    \label{fig:circle}}
    \hfill
        \subfloat[The three-petal rose]{
        \begin{tikzpicture}[scale=.75]
        \draw[black, thick] (0,0) to[out=-30,in=270] (2,0) to[out=90,in=30] (0,0) to[out=90,in=30] (-1,1.73) to[out=210,in=150] (0,0) to[out=210,in=150] (-1,-1.73) to[out=-30,in=270] (0,0);
        \filldraw[black] (2,0) circle (2pt) node[anchor=west]{$p$};
        \filldraw[black] (-1,1.73) circle (2pt) node[anchor=south]{$q$};
        \filldraw[black] (-1,-1.73) circle (2pt) node[anchor=north]{$r$};
        \filldraw[black] (0,0) circle (2pt);
        \end{tikzpicture}
    \label{fig:rose}}
    \hfill \hfill

    \hfill
    \subfloat[A tripod]{
        \begin{tikzpicture}[scale=.75]
        \draw[black, thick] (0,0)--(2,0);
        \draw[black, thick] (-1,1.73)--(0,0) -- (-1,-1.73);
        \filldraw[black] (0,0) circle (2pt);        
        \filldraw[black] (2,0) circle (2pt) node[anchor=north]{$p$};
        \filldraw[black] (-1,1.73) circle (2pt) node[anchor=east]{$q$};
        \filldraw[black] (-1,-1.73) circle (2pt) node[anchor=east]{$r$};
        \end{tikzpicture}
    \label{fig:tripod}}
    \hfill
    \subfloat[The twisted heart]{
             \begin{tikzpicture}[scale=.75]
        \draw[black, thick] (0,-2)--(1.73,1) to[out=180,in=30] (-.5,.5) -- (0,-.37);
        \filldraw[white] (0,.7) circle (4pt);
        \draw[black, thick]  (0,-.37) -- (.5,.5) to[out=150,in=0] (-1.73,1) -- (0,-2);
        \draw[red] (.865,-.5) -- (.5,.5);
        \draw[red] (-.865,-.5) -- (-.5,.5);
        \filldraw[black] (0,-2) circle (2pt) node[anchor=north]{$p$};
        \filldraw[black] (1.73,1) circle (2pt) node[anchor=west]{$q$};
        \filldraw[black] (-1.73,1) circle (2pt) node[anchor=east]{$r$};
        \filldraw[black] (.4325,-1.25) circle (2pt);
        \filldraw[black] (.865,-.5) circle (2pt) node[anchor=west]{$a$};
        \filldraw[black] (1.3,.25) circle (2pt);
        \filldraw[black] (-.4325,-1.25) circle (2pt);
        \filldraw[black] (-.865,-.5) circle (2pt) node[anchor=east]{$b$};
        \filldraw[black] (-1.3,.25) circle (2pt);
        \filldraw[black] (-.5,.5) circle (2pt) node[anchor=east]{$e$};
        \filldraw[black] (0,-.37) circle (2pt);
        \filldraw[black] (.5,.5) circle (2pt) node[anchor=west]{$c$};
        \end{tikzpicture}
    \label{fig:heart}}
    \hfill \hfill
    \caption{Examples of metric polygons.}
    \label{fig:examples}
\end{figure}

Our final definition, used in the proof of \Cref{thm:embedding}, is the Gromov product, which measures the deficit in the triangle inequality for a triple of points.
\begin{defi}
    Let $X$ be a metric space and $p,q,r \in X$. The \textit{Gromov product} of $p$ and $q$ at $r$ is 
    \[(p \cdot q)_r = \frac{1}{2}(d(p,r)+d(q,r)-d(p,q)).\]
\end{defi}
The usefulness of the Gromov product for our purpose is that it allows us to associate to any metric triangle a tripod with matching edge lengths. This idea forms the basis of our proof that any metric triangle is bi-Lipschitz embeddable in $\mathbb{R}^2$. The following simple lemma was proved as Lemma 3.1 in \cite{NR:23}.
\begin{lemm} \label{lemm:tripod}
    Let $\triangle$ be a metric triangle with vertices $p,q,r$. Let $T$ be the tripod formed as in \Cref{exm:examples} by taking $l_1 = (q \cdot r)_p$, $l_2 = (p \cdot r)_q$ and $l_3 = (p \cdot q)_r$. Define a map $P\colon \triangle \to T$ by setting $P(p) = l_1 \in T_1$, $P(q) = l_2 \in T_2$, $P(r) = l_3 \in T_3$, and mapping each edge of $\triangle$ by arc length onto the corresponding edge of $T$. Then $P$ is $1$-Lipschitz. 
\end{lemm}

Each tripod $T$ can in turn be embedded into the plane in a natural way. For a given tripod $T$ with vertices $p,q,r$ and center $o$, let $r_1 = d(o,p)$, $r_2 = d(o,q)$, $r_3 = d(o,r)$. Let $u_1 = (r_1,0)$, $u_2 = r_2(-\frac{1}{2}, \frac{\sqrt{3}}{2})$, and $u_3 = r_3(-\frac{1}{2}, -\frac{\sqrt{3}}{2})$. For each $i = 1,2,3$, let $J_i$ be the union of the straight line segment from $(0,0)$ to $u_i$. Then let $Y = J_1 \cup J_2 \cup J_3$, equipped with the Euclidean metric from $\mathbb{R}^2$. We define a map $\varphi \colon T \to Y$ by mapping the segment $[op]$ onto $J_1$ by arc length, $[oq]$ onto $J_2$ by arc length and $[or]$ onto $J_3$ by arc length. 

    \section{Proof of optimal distortion bound} \label{sec:embedding} 

        In this section, we give the proof of our main result, \Cref{thm:embedding}, giving the optimal distortion bound for the tripodal embedding of an arbitrary metric triangle. 

        We start by reviewing the definition of the tripodal embedding as given in \cite{NR:23}. Let $\triangle$ be an arbitrary metric triangle with edges $I_1, I_2, I_3$. Denote the metric on $\triangle$ by $d$. Let $p$ be the common vertex of $I_1$ and $I_3$, $q$ be the common vertex of $I_1$ and $I_2$,  and $r$ be the common vertex of $I_2$ and $I_3$. For each edge $I_j$, let $\widehat{I}_j$ denote the union of the other two edges of $\triangle$. We can associate to the triangle $\triangle$ a metric tripod $T$ with matching edge lengths and corresponding map $P \colon \triangle \to T$ as in \Cref{lemm:tripod}. To this $T$ we associate an embedded tripod $Y$ and map $\varphi \colon T \to Y$ following the notation of the previous section. For conciseness, we write $\bar{x}$ in place of $(\varphi \circ P)(x)$.
      Let $v_1 = (\frac{1}{2}, \frac{\sqrt{3}}{2})$, $v_2 = (-1,0)$ and $v_3 = (\frac{1}{2}, -\frac{\sqrt{3}}{2})$. Then $F \colon \triangle \to \mathbb{R}^2$ is defined by the formula
        \[ F(x) = \bar{x} + \dist(x, \widehat{I}_j)v_j \text{ if } x \in I_j.\]

        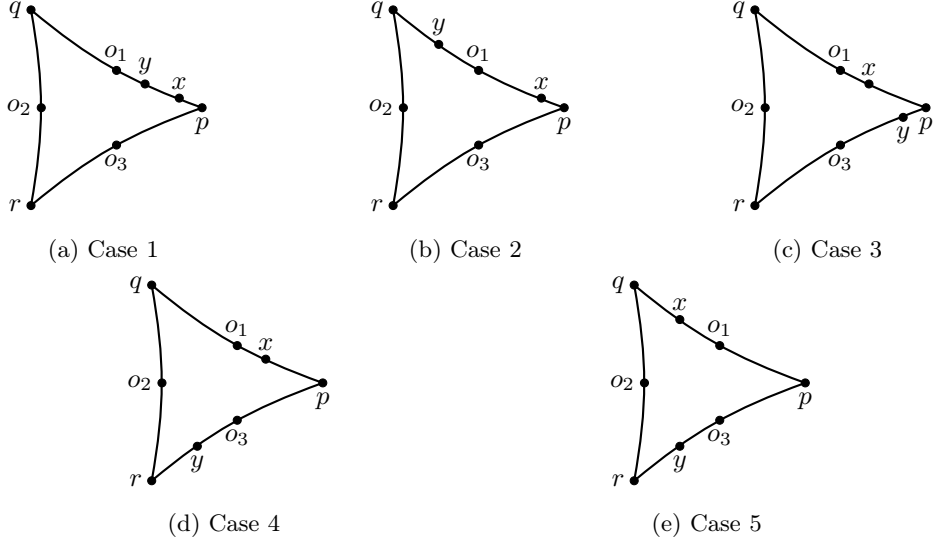
\begin{figure} 
    \centering
    \hfill
    \subfloat[Case 1]{
        \begin{tikzpicture}[scale=.75]
        \draw[black, thick] (2,0) to[out=160,in=320]  (-1,1.73) to[out=280,in=80] (-1,-1.73) to[out=40,in=200] (2,0);        
        \filldraw[black] (2,0) circle (2pt) node[anchor=north]{$p$};
        \filldraw[black] (-1,1.73) circle (2pt) node[anchor=east]{$q$};
        \filldraw[black] (-1,-1.73) circle (2pt) node[anchor=east]{$r$};
        \filldraw[black] (.5,.66) circle (2pt) node[anchor=south]{$o_1$};
        \filldraw[black] (-.82,0) circle (2pt) node[anchor=east]{$o_2$};
        \filldraw[black] (.5,-.66) circle (2pt) node[anchor=north]{$o_3$};
        \filldraw[black] (1,.42) circle (2pt) node[anchor=south]{$y$};
        \filldraw[black] (1.6,.17) circle (2pt) node[anchor=south]{$x$};
        \end{tikzpicture}
    \label{fig:five_cases_a}}
    \hfill
    \subfloat[Case 2]{
             \begin{tikzpicture}[scale=.75]
        \draw[black, thick] (2,0) to[out=160,in=320]  (-1,1.73) to[out=280,in=80] (-1,-1.73) to[out=40,in=200] (2,0);        
        \filldraw[black] (2,0) circle (2pt) node[anchor=north]{$p$};
        \filldraw[black] (-1,1.73) circle (2pt) node[anchor=east]{$q$};
        \filldraw[black] (-1,-1.73) circle (2pt) node[anchor=east]{$r$};
        \filldraw[black] (.5,.66) circle (2pt) node[anchor=south]{$o_1$};
        \filldraw[black] (-.82,0) circle (2pt) node[anchor=east]{$o_2$};
        \filldraw[black] (.5,-.66) circle (2pt) node[anchor=north]{$o_3$};
        \filldraw[black] (1.6,.17) circle (2pt) node[anchor=south]{$x$};
        \filldraw[black] (-.2,1.12) circle (2pt) node[anchor=south]{$y$};
        \end{tikzpicture}
    \label{fig:five_cases_b}}
    \hfill\subfloat[Case 3]{
             \begin{tikzpicture}[scale=.75]
        \draw[black, thick] (2,0) to[out=160,in=320]  (-1,1.73) to[out=280,in=80] (-1,-1.73) to[out=40,in=200] (2,0);        
        \filldraw[black] (2,0) circle (2pt) node[anchor=north]{$p$};
        \filldraw[black] (-1,1.73) circle (2pt) node[anchor=east]{$q$};
        \filldraw[black] (-1,-1.73) circle (2pt) node[anchor=east]{$r$};
        \filldraw[black] (.5,.66) circle (2pt) node[anchor=south]{$o_1$};
        \filldraw[black] (-.82,0) circle (2pt) node[anchor=east]{$o_2$};
        \filldraw[black] (.5,-.66) circle (2pt) node[anchor=north]{$o_3$};
        \filldraw[black] (1,.42) circle (2pt) node[anchor=south]{$x$};
        \filldraw[black] (1.6,-.17) circle (2pt) node[anchor=north]{$y$};
        \end{tikzpicture}
    \label{fig:five_cases_c}}
    \hfill \hfill

    \hfill
    \subfloat[Case 4]{
        \begin{tikzpicture}[scale=.75]
        \draw[black, thick] (2,0) to[out=160,in=320]  (-1,1.73) to[out=280,in=80] (-1,-1.73) to[out=40,in=200] (2,0);        
        \filldraw[black] (2,0) circle (2pt) node[anchor=north]{$p$};
        \filldraw[black] (-1,1.73) circle (2pt) node[anchor=east]{$q$};
        \filldraw[black] (-1,-1.73) circle (2pt) node[anchor=east]{$r$};
        \filldraw[black] (.5,.66) circle (2pt) node[anchor=south]{$o_1$};
        \filldraw[black] (-.82,0) circle (2pt) node[anchor=east]{$o_2$};
        \filldraw[black] (.5,-.66) circle (2pt) node[anchor=north]{$o_3$};
        \filldraw[black] (1,.42) circle (2pt) node[anchor=south]{$x$};
        \filldraw[black] (-.2,-1.12) circle (2pt) node[anchor=north]{$y$};
        \end{tikzpicture}
    \label{fig:five_cases_d}}
    \hfill
    \subfloat[Case 5]{
             \begin{tikzpicture}[scale=.75]
        \draw[black, thick] (2,0) to[out=160,in=320]  (-1,1.73) to[out=280,in=80] (-1,-1.73) to[out=40,in=200] (2,0);        
        \filldraw[black] (2,0) circle (2pt) node[anchor=north]{$p$};
        \filldraw[black] (-1,1.73) circle (2pt) node[anchor=east]{$q$};
        \filldraw[black] (-1,-1.73) circle (2pt) node[anchor=east]{$r$};
        \filldraw[black] (.5,.66) circle (2pt) node[anchor=south]{$o_1$};
        \filldraw[black] (-.82,0) circle (2pt) node[anchor=east]{$o_2$};
        \filldraw[black] (.5,-.66) circle (2pt) node[anchor=north]{$o_3$};
        \filldraw[black] (-.2,1.12) circle (2pt) node[anchor=south]{$x$};
        \filldraw[black] (-.2,-1.12) circle (2pt) node[anchor=north]{$y$};
        \end{tikzpicture}
    \label{fig:five_cases_e}}
    \hfill \hfill
    \caption{The five cases in the proof of \Cref{thm:embedding}.}
    \label{fig:five_cases}
\end{figure}

        Here, given a point $x \in X$ and a set $A \subset X$, we write $\dist(x,A) = \inf_{a \in A} d(x,y)$. 

        We fix some additional notation. Let $I_1^p$ denote the subarc of $I_1$ consisting of those points $x \in I_1$ for which $d(x,p) \leq (q \cdot r)_p$. Define $I_1^q$, $I_2^q$, $I_2^r$, $I_3^r$ and $I_3^p$ similarly. For each edge $I_j$, let $o_j$ denote the unique point such that $\bar{o}_j$ is the origin. Observe that $I_1 = I_1^p \cup I_1^q$, with $I_1^p$ and $I_1^q$ intersecting at the point $o_1$, and likewise for $I_2$ and $I_3$.  

        In this proof, we consider two arbitrary distinct points $x,y \in \triangle$ and show that $\|F(x) - F(y)\|$ satisfies appropriate bounds. By symmetry, we may consider five individual cases in our proof of \Cref{thm:embedding}: (1) $x,y \in I_1^p$; (2) $x \in I_1^p, y \in I_1^q$; (3) $x \in I_1^p, y \in I_3^p$; (4) $x \in I_1^p, y \in I_3^r$; (5) $x \in I_1^q, y \in I_3^r$. See \Cref{fig:five_cases}. Each of these cases is represented by one of the lemmas in this section. Next, given this choice of $x$ and $y$, we let $z \in \triangle$ denote a point satisfying $d(x,z) = \dist(x,\widehat{I}_1)$. This point may fail to be uniquely determined, so we fix a choice of $z$. Similarly, we let $w \in \triangle$ denote a point satisfying $d(y,w) = \dist(y,\widehat{I}_j)$, where $j$ is such that $y \in I_j$.

        \begin{lemm} \label{lemm:distortion_case1}
            Suppose that $x,y \in I_1^p$. Then
            \[  \frac{\sqrt{3}}{2}d(x,y) \leq \|F(x) - F(y)\| \leq \sqrt{3}d(x,y).\]
        \end{lemm} 
        \begin{proof}
            Assume that $d(x,p) < d(y,p)$. We first verify the lower bound. Note that both $\Bar{x}$ and $\Bar{y}$ lie on the positive real axis, with $\|\Bar{x}-\Bar{y}\| = d(x,y)$. Moreover, $F(x)$ and $F(y)$ lie on the lines $\{\bar{x} + t v_1: t \in \mathbb{R}\}$ and $\{\bar{y} + t v_1: t \in \mathbb{R}\}$, respectively. The distance between these lines is $\frac{\sqrt{3}}{2}d(x,y)$. See \Cref{fig:case1_lower}. This verifies the lower bound.

            Next, we show the upper bound. Observe that $|d(x,\widehat{I}_1) - d(y,\widehat{I}_2)| \leq d(x,y)$. Write $F(y)$ as $F(y) = \bar{y} + t(y)v_1$ for some $t(y) \geq 0$. Then $F(x)$ belongs to the set $\bar{x} + t v_j$, where $t \in [t(y) - d(x,y), t(y) + d(x,y)]$. The point in this set furthest from $F(y)$ is $\bar{x} + (t(y) + d(x,y))v_j$, which is at distance $\sqrt{3}d(x,y)$ from $F(y)$. See \Cref{fig:case1_upper}. This verifies the upper bound.
        \end{proof}

                \begin{figure} 
    \centering
    \hfill
    \subfloat[Lower bound in \Cref{lemm:distortion_case1}]{
        \begin{tikzpicture}[scale=.75]
        \draw[gray, thick] (0,0)--(-2,3.46);
        \draw[gray, thick] (-2.3/2,-3.46/2)--(0,0) -- (5,0);
        \draw[gray, dotted] (1,0) -- (3, 3.46);
        \draw[gray, dotted] (3,0) -- (5,3.46);
        \draw[gray, thick] (3,0) -- (1.5,.87);
        \filldraw[black] (3,0) circle (2pt) node[anchor=north]{$F(x)$};
        \filldraw[black] (1,0) circle (2pt) node[anchor=north]{$\Bar{y}$};
        \filldraw[black] (1.5,.87) circle (2pt) node[anchor=east]{$F(y)$};
        \filldraw[black] (0,0) circle (2pt) node[anchor=east]{$o$};
        \end{tikzpicture}
    \label{fig:case1_lower}}
    \hfill
    \subfloat[Upper bound in \Cref{lemm:distortion_case1}]{
             \begin{tikzpicture}[scale=.75]
        \draw[gray, thick] (0,0)--(-2,3.46);
        \draw[gray, thick] (-2.3/2,-3.46/2)--(0,0) -- (5,0);
        \draw[gray, dotted] (1,0) -- (3, 3.46);
        \draw[gray, dotted] (3,0) -- (5,3.46);
        \draw[gray, thick] (1.5,.87) -- (3.5, .87) -- (4.5,2.6) -- (1.5, .87);
        \filldraw[black] (4.5,2.6) circle (2pt) node[anchor=west]{$F(x)$};
        \filldraw[black] (1,0) circle (2pt) node[anchor=north]{$\Bar{y}$};
        \filldraw[black] (1.5,.87) circle (2pt) node[anchor=east]{$F(y)$};
        \filldraw[black] (3,0) circle (2pt) node[anchor=north]{$\Bar{x}$};
        \filldraw[black] (0,0) circle (2pt) node[anchor=east]{$o$};
        \end{tikzpicture}
    \label{fig:case1_upper}}
    \hfill \hfill
    \caption{}
    \label{fig:case1}
\end{figure}

        \begin{lemm} \label{lemm:distortion_case2}
               Suppose that $x \in I_1^p$ and $y \in I_1^q$. Then
            \[  \frac{\sqrt{3}}{2}d(x,y) \leq \|F(x) - F(y)\| \leq \sqrt{3}d(x,y).\]
        \end{lemm}
        \begin{proof}
            Assume without loss of generality that $d(o_1,y) \leq d(o_1,x)$. We first verify the lower bound. As in the previous case, observe that $F(x)$ and $F(y)$ lie on the lines $\{\bar{x} + tv_1: t \in \mathbb{R}\}$ and $\{\bar{y} + tv_1: t \in \mathbb{R}\}$, respectively. The distance between these lines is $\frac{\sqrt{3}}{2}d(o_1,x) + \frac{\sqrt{3}}{2}d(o_1,y)$. It follows that
            \[\frac{\sqrt{3}}{2}(d(o_1,x) + d(o_1,y)) \leq \|F(x) - F(y)\|.\] See \Cref{fig:case2_lower}. Since $d(x,y) = d(x,o_1) + d(o_1,y)$, we obtain the lower bound. 

            We next verify the upper bound. Write $F(y)$ as $F(y) = \bar{y} + t(y)v_1$ for some $t(y) \geq 0$. As in the previous case, $F(x)$ belongs to the set $\bar{x} + tv_1$, where $t \in [t(y) - d(x,y), t(y) + d(x,y)]$. Since $d(o_1,y) \leq d(o_1,x)$, the point $\bar{x} + (t(y) + d(x,y))v_1$ has largest possible distance from $F(y)$ in this set. Denote the latter point by $\widetilde{x}$. Let $\widetilde{y} = (-d(o_1,y),0) + t(y)v_1$; that is, in the definition of $F(y)$, we replace $\bar{y}$ with the point along the negative horizontal axis at the same distance from the origin. The point $\widetilde{y}$ belongs to the same line $\{\bar{y} + tv_1: t \in \mathbb{R}\}$ but is further away from $\widetilde{x}$ than $F(y)$. Thus $\|F(x) - F(y)\| \leq \|\widetilde{x} - \widetilde{y}\|$. But $\|\widetilde{x} - \widetilde{y}\| \leq \sqrt{3}d(x,y)$ by the same analysis as in \Cref{lemm:distortion_case1}. This establishes the upper bound. 
        \end{proof}

                \begin{figure} 
    \centering
    \hfill
    \subfloat[The lower bound in \Cref{lemm:distortion_case2} is the length of the solid line shown]{
        \begin{tikzpicture}[scale=.75]
        \draw[gray, thick] (0,0)--(-2,3.46);
        \draw[gray, thick] (-2.3/2,-3.46/2)--(0,0) -- (5,0);
        \draw[gray, dotted] (3,0) -- (4.5,3.46);
        \draw[gray, dotted] (-1,1.73) -- (0,3.46);
        \draw[gray, thick] (3,0) -- (-.75,2.17);
        \filldraw[black] (-1,1.73) circle (2pt) node[anchor=east]{$F(y)$};
        \filldraw[black] (3,0) circle (2pt) node[anchor=north]{$F(x)$};
        \filldraw[black] (0,0) circle (2pt) node[anchor=east]{$o$};
        \filldraw[black] (-.75,2.17) circle (2pt);
        \end{tikzpicture}
    \label{fig:case2_lower}}
    \hfill
    \subfloat[The upper bound in \Cref{lemm:distortion_case2} is the length of the solid line shown]{
             \begin{tikzpicture}[scale=.75]
        \draw[gray, thick] (0,0)--(-2,3.46);
        \draw[gray, thick] (-2.3/2,-3.46/2)--(0,0) -- (5,0);
        \draw[gray, dotted] (4/3,0) -- (10/3,3.46);
        \draw[gray, dotted] (-4/3,0) -- (2/3,3.46);
        \draw[gray, thick] (-4/3,0) -- (2.67,2.3);
        \draw[gray, dotted] (0,0) -- (-2,0);
        \filldraw[black] (-.67,1.15) circle (2pt) node[anchor=east]{$F(y)$};
        \filldraw[black] (4/3,0) circle (2pt) node[anchor=north]{$\bar{x}$};
        \filldraw[black] (-4/3,0) circle (2pt) node[anchor=north]{$\widetilde{y}$};
        \filldraw[black] (2.27,1.6) circle (2pt) node[anchor=west]{$F(x)$};
        \filldraw[black] (2.67,2.3) circle (2pt) node[anchor=west]{$\widetilde{x}$};
        \filldraw[black] (0,0) circle (2pt) node[anchor=east]{$o$};
        \end{tikzpicture}
    \label{fig:case2_upper}}
    \hfill \hfill
    \caption{}
    \label{fig:case2}
    \end{figure}

    In the remaining lemmas, it is more convenient to phrase the inequalities in terms of the square of the distortion rather than the distortion itself. We introduce some more notation. Note that $x \in I_1$ and $y \in I_3$ in all the remaining cases. We let $a = d(x,z)$, $b = d(y,w)$, $s = d(x,o_1)$ and $t = d(y,o_3)$. Recall that $z$ be a point in $\widehat{I}_1$ such that $d(x,z) = d(x,\widehat{I}_1)$, and $w$ a point in $\widehat{I}_3$ such that $d(y,w) = d(y,\widehat{I}_3)$. 

        \begin{lemm} \label{lemm:distortion_case3}
            Suppose that $x \in I_1^p$ and $y \in I_3^p$. Then 
        \[\frac{3}{16} \leq \frac{\|F(x) - F(y)\|^2}{d(x,y)^2} \leq 4.\]
        \end{lemm}
        \begin{proof}
            Observe that $\bar{x}$ and $\bar{y}$ belong to the same leg of the embedded tripod $Y$. 
            First we show the upper bound. Note that $d(x,y) \geq |t-s|$, since by \Cref{lemm:tripod} the map $\varphi \circ P$ is $1$-Lipschitz. It follows that $d(x,y) \geq \max\{a,b,|t-s|\}$. Let $e = |t-s|$. 
            Observe that $F(x) = (s + \frac{1}{2}a, \frac{\sqrt{3}}{2}a)$ and $F(y) = (t + \frac{1}{2}b, - \frac{\sqrt{3}}{2}b)$, from which we calculate
\[\|F(x) - F(y)\|^2 = \frac{1}{4}\left(a - b + 2s - 2t \right)^2 + \frac{3}{4}\left(a + b\right)^2.\]
By rescaling and symmetry, we assume now that $a=1$ and $b \leq 1$. If $|e| \leq 1$, then $d(x,y) \geq \max\{a,b\} = 1$ and hence 
\[\frac{\|F(x) - F(y)\|^2}{d(x,y)^2} \leq M_1(b,e)\]
for
\[M_1(b,e) = \frac{1}{4}\left(1 - b + 2e \right)^2 + \frac{3}{4} \left(1 + b\right)^2.\]
The function $M_1(b,e)$ is maximized on the set $\{(b,e): 0 \leq b,e \leq 1\}$ by taking $b=e=1$, which gives a value of $4$. Since this statement is plausible enough (for example, by plotting the function), and checking it is unrelated to the rest of the proof, we defer the verification of this claim to the Appendix. See \Cref{lemm:M1} for the details. We do the same with similar inequalities arising in this section. 

If $|e| \geq 1$, then 
\[\frac{\|F(x) - F(y)\|^2}{d(x,y)^2} \leq \frac{1}{e^2}M_1(b,e).\]
On the set $\{(b,e): 0 \leq b \leq 1, 1 \leq e <\infty\}$, the function $e^{-2}M_1(b,e)$ is maximized by again taking $b= e= 1$, which gives a value of $4$. See \Cref{lemm:M1b}. This establishes the upper bound. 

Now we show the lower bound. At this point, we separate into two cases. \newline 

\noindent \textit{Case 1.} Assume first that $z \in I_3$ or $w \in I_1$. By symmetry, we may assume that $z \in I_3$. If $d(p,z) \leq d(p,y)$, the triangle inequality implies that
\begin{align*}
  d(p,o_1) & \leq  d(p,z) + d(z,x) + d(x,o_1)  \\
        & = d(p,z) + a + s  = d(p,o_3) - d(o_3,y) - d(y,z) + a + s\\
        & = d(p,o_3) - t - d(y,z) + a + s. 
\end{align*}
Since $d(p,o_1) = d(p,o_3)$, this rearranges to $d(y,z) \leq a + s - t$, from which we conclude that
\[d(x,y) \leq d(y,z) + d(z,x) \leq 2a + s - t.\]
Assume without loss of generality that $a=1$. Set $e = s-t$. Then $d(x,y)^2/\|F(x) - F(y)\|^2$ is bounded by the function
\[M_2(b,e) = \frac{(2 + e)^2}{\frac{1}{4}\left(1 - b + 2e\right)^2 + \frac{3}{4}(1+b)^2}. \]
In \Cref{lemm:M2}, we show that $M_2(b,e)$ is bounded on the set $\{(b,e): b \geq 0\}$ by $4$. 

On the other hand, if $d(p,y) \leq d(p,z)$, the triangle inequality implies that
\begin{align*}
  d(p,o_3) & \leq d(o_3,z) + d(z,x) + d(x,p) \\ 
     & = d(o_3,z) +a + d(x,p) = d(o_3,y) - d(y,z) + a + d(p,o_1) - d(o_1,y) \\
     & = t - d(z,y) + a + d(p,o_1) - s,  
\end{align*}
which gives $d(z,y) \leq a + t-s$ and hence 
\[d(x,y) \leq d(y,z) + d(z,x) \leq 2a + t - s.\]
In this case, $d(x,y)^2/\|F(x) - F(y)\|^2$ is bounded by the function 
\[\widetilde{M}_2(b,e) = \frac{(2 - e)^2}{\frac{1}{4}\left(1 - b + 2e\right)^2 + \frac{3}{4}(1+b)^2}. \]
If $s \geq t$, then $e \geq 0$ and we have the relation $\widetilde{M}_2(b,e) \leq M_2(b,e)$. If $s < t$, we claim that one of two possibilities necessarily holds: either $b \geq t-s$, or $w \in I_1$ with $d(p,w) \leq d(p,x)$. If $w \in I_2$, then necessarily $b \geq t$ and hence the first alternative holds. If $w \in I_1$ with $d(p,w) > d(p,x)$, then the triangle inequality gives 
\[d(p,o_1) \leq d(p,y) + d(y,w) + d(w,o_1) = d(p,o_3) - t + b + d(w,o_1),\]
and hence $t - d(w,o_1) \leq b$. Since $d(w,o_1) < s$, this establishes the claim. 

By \Cref{lemm:M2b}, the function $\widetilde{M}_2(b,e)$ is bounded on the set $\{(b,e): b \geq 0, -b \leq e \leq 0\}$ by $13/3$. This covers the case where $b \geq t-s$.  



Assume next that $w \in I_1$ with $d(p,w) \leq d(p,x)$. Since $w \in I_1$, we can follow the previous argument with roles of $x$ and $y$ reversed. Since $d(p,w) \leq d(p,x)$, we get the inequality $d(x,y) \leq 2b + t - s$ analogously to the first alternative above, which is handled by \Cref{lemm:M2}. \newline  

\noindent \textit{Case 2.} As the second case, assume now that both $w,z \in I_2$. Let $d_1 = d(r,w) - d(r,o_2)$ and let $d_2 = d(q,z) - d(q,o_2)$. Note that $d_1$ and $d_2$ may be negative. The triangle inequality yields the inequalities
\[ d(p,q) \leq d(p,y) + d(y,w) + d(w,q) = d(p,o_3) - t + b + d(o_2,q) - d_1 \]
and
\[ d(p,r) \leq d(p,y) + d(y,w) + d(w,r) = d(p,o_3) - t + b + d(o_2,r) + d_1. \]
If $d_1 \geq 0$, we use the first inequality and the relation $d(p,q) = d(p,o_3) + d(o_2,q)$ to conclude that $d_1 \leq b-t$. If $d_1 \leq 0$, we use the second inequality and the relation $d(p,r) = d(p,o_3) + d(o_2,r)$ to conclude that $-d_1 \leq b-t$. Together, these show that $|d_1| \leq b-t$. Similarly, we have $|d_2| \leq a-s$. 

This gives 
\[d(z,w) \leq |d_1| + |d_2| \leq a + b - s-t\]
and hence \[d(x,y) \leq d(x,z) + d(z,w) + d(w,y) \leq a + b + |d_1| + |d_2| \leq  2a + 2b - s-t.\] 

We compute
\[\frac{d(x,y)^2}{\|F(x) - F(y)\|^2} \leq \frac{(2a+2b- s-t)^2}{\frac{1}{4}(a - b + 2s -2t)^2 + \frac{3}{4}(a + b)^2} .\]
Assume without loss of generality that $a=1$ and $b \leq 1$. Then $d(x,y)^2/\|F(x) - F(y)\|^2$ is bounded by the function
\[M_3(b,s,t) = \frac{(2+2b- s-t)^2}{\frac{1}{4}(1 - b + 2s -2t)^2 + \frac{3}{4}(1 + b)^2}.\]
Note as well that $s \leq a=1$ and $t \leq b \leq 1$. Thus we must maximize the function $M_3(b,s,t)$ on the set $0 \leq b,s,t \leq 1$. We show in \Cref{lemm:M3} that the maximum value is $16/3$, which is attained when $b=1$ and $s=t=0$.
        \end{proof}

        The remaining two cases are handled similarly to \Cref{lemm:distortion_case3}, and so our proofs omit some details. 

        \begin{lemm} \label{lemm:distortion_case4}
            Suppose that $x \in I_1^p$ and $y \in I_3^r$. Then 
        \[\frac{3}{16} \leq \frac{\|F(x) - F(y)\|^2}{d(x,y)^2} \leq 7.\]
        \end{lemm}
        \begin{proof}
            We begin with the upper bound. Observe as before that $d(x,y) \geq \max\{a,b,s+t\}$ and that $F(x) = (s + \frac{1}{2}a, \frac{\sqrt{3}}{2}a)$. On the other hand, now $F(y) = (\frac{1}{2}(-t+b), -\frac{\sqrt{3}}{2} (t+b))$, and we obtain
            \[\|F(x) - F(y)\|^2 = \frac{1}{4}\left(a - b + 2s  + t\right)^2 + \frac{3}{4}\left(a + t + b\right)^2 .\]
            Assume now without loss of generality that $a= 1$. If $\max\{a,b,s+t\} = a = 1$, then $b \leq 1$ and $s + t \leq 1$ and so
    \[\frac{\|F(x) - F(y)\|^2}{d(x,y)^2} \leq M_4(b,s,t),\]
    where 
    \[M_4(b,s,t) = \frac{1}{4}\left(1 - b + 2s + t\right)^2 + \frac{3}{4}\left(1 + t + b \right)^2. \]
    Over the region $\{(b,s,t): 0 \leq b,s,t \leq 1, s+t \leq 1\}$, the expression on the right of the previous inequality is maximized by taking $s = 0$, $b = t = 1$ by \Cref{lemm:M4}. This gives the desired upper bound $1/4 + 3/4 \cdot 3^2 = 7$. 

    If $\max\{a,b,s + t\} = s+t$, then
    \[\frac{\|F(x) - F(y)\|^2}{d(x,y)^2} \leq \frac{M_4(b,s,t)}{(s+t)^2}. \]
    Over the region $\{(b,s,t): 0 \leq s,t \leq 1, s+t \geq 1, 0 \leq b \leq s+t\}$, the expression on the right of the previous inequality is again maximized by taking $s = 0$, $b = t = 1$ by \Cref{lemm:M4b}.

    If $\max\{a,b,s+t\} = b$, then 
    \[\frac{\|F(x) - F(y)\|^2}{d(x,y)^2} \leq \frac{M_4(b,s,t)}{b^2}. \]
    Over the region $\{(b,s,t): 0 \leq s,t \leq 1, b \geq 1, s + t\leq b\}$, the expression on the right of the previous inequality is again maximized by taking $s = 0$, $b=t = 1$ by \Cref{lemm:M4c}.

    Next, we show the lower bound. We separate into three cases. \newline

    \noindent \textit{Case 1.}
    Assume first that $z \in I_3$. We claim that $d(x,y) \leq 2a + s + t$. If $d(z,p) \leq d(y,p)$, then we bound $d(x,y)$ as follows. By the triangle inequality, 
    \[d(o_1,p) = d(o_1,x) + d(x,p) = s + d(x,p) \leq s + a + d(z,p).\]
    Since $d(z,p) = d(o_3,p) - d(o_3,z) = d(o_1,p) - d(o_3,z)$, this gives
    \[d(o_3,z) \leq a+s.\]
    Thus $d(x,y) \leq t + d(o_3,z) + a \leq 2a + s + t$. 
    If $d(z,p) > d(y,p)$, we obtain the same bound as follows. By the triangle inequality, 
    \[d(r,p) = d(r,z) + d(z,y) + t + d(o_3,p) \leq d(r,z) + a + d(x,p).\]
    Since $d(o_3,p) = d(o_1,p) = d(x,p) + s$, the previous inequality implies that $d(z,y) \leq a - s - t$ and hence $d(x,y) \leq 2a - s -t \leq 2a + s + t$. It follows that
    \[\frac{d(x,y)^2}{\|F(x) - F(y)\|^2} \leq \frac{(2a + s + t)^2}{\frac{1}{4}\left(a - b + 2s  + t\right)^2 + \frac{3}{4}\left(a + t + b\right)^2}.\] 
    A bound on the right-hand side quantity will be found together with the next case. \newline

    \noindent \textit{Case 2.}
    Assume in the second case that $w \in I_1$. We claim now that $d(x,y) \leq 2b + s + t$. If $d(w,q) \leq d(o_1,q)$, then we have
    \[d(q,r) \leq d(y,q) + b + d(w,r) = d(r,o_3) - t + b + d(q,o_1) - d(w,o_1).\]
    This implies that $t + d(w,o_1) \leq b$. Then 
    \[d(x,y) \leq b + d(w,o_1) + s \leq 2b + s -t.\]
    Next, if $w$ lies between $o_1$ and $x$, then $d(x,y) \leq b+ s \leq 2b+s-t$, since $b \geq t$. If $w$ lies between $x$ and $p$, then \[d(p,r) \leq d(p,o_3) - t + b + d(w,p) \leq d(p,o_3) - t + b + d(o_1,p) - s - d(x,w),\]
    which implies $d(x,y) \leq d(x,w) + d(w,y) \leq 2b - s-t \leq 2b - s + t$.  
    
    It follows in each case that $d(x,y) \leq 2b+s+t$ and hence that
    \[\frac{d(x,y)^2}{\|F(x) - F(y)\|^2} \leq \frac{(2b + s + t)^2}{\frac{1}{4}\left(a - b + 2s  + t\right)^2 + \frac{3}{4}\left(a + t + b\right)^2}.\]

    Without loss of generality, assume that $a=1$. Note as well that $s \leq a =1$. In both of the first two cases, $d(x,y)^2/\|F(x) - F(y)\|^2$ is bounded by
    \[M_5(b,s,t) = \frac{(\max\{2,2b\} + s + t)^2}{\frac{1}{4}\left(1 - b + 2s  + t\right)^2 + \frac{3}{4}\left(1 + t + b\right)^2}.\]
    By \Cref{lemm:M5}, the function $M_5(b,s,t)$ is bounded by $13/3$ on the set $\{(b,s,t): 0 \leq b,s,t, 1 \geq s\}$. \newline

    \noindent \textit{Case 3.}
    In the final case, we assume that $z \in I_2$ and $w \in I_2$. Proceeding similarly to Case 2 of \Cref{lemm:distortion_case3}, we obtain the relation $d(x,y) \leq 2a+2b-s+t$. Assume without loss of generality that $a=1$. Then $d(x,y)^2/\|F(x) - F(y)\|^2 \leq M_6(b,s,t)$, where
    \[M_6(b,s,t) = \frac{(2+2b+2t)^2}{\frac{1}{4}\left(1 - b + 2s  + t\right)^2 + \frac{3}{4}\left(1 + t + b\right)^2} .\]
    Note that the numerator of this formula uses the weaker bound $2+2b+2t$ instead of $2+2b-s+t$, but this turns out to simplify the calculations. By \Cref{lemm:M6}, this is bounded on the set $\{(b,s,t): 0 \leq b,s,t\}$ by $16/3$. 
        \end{proof}

        \begin{lemm} \label{lemm:distortion_case5}
            Suppose that $x \in I_1^q$ and $y \in I_3^r$. Then 
        \[\frac{3}{16} \leq \frac{\|F(x) - F(y)\|^2}{d(x,y)^2} \leq 7.\]
        \end{lemm}
        \begin{proof}
            We prove the upper bound. This is similar to the proof of \Cref{lemm:distortion_case4}. As before, we have the lower bound $d(x,y) \geq \max\{a,b,s+t\}$. However, in this case we have $F(x) = (\frac{1}{2}(-s+a), \frac{\sqrt{3}}{2} (s+a))$ $F(y) = (\frac{1}{2}(-t+b), -\frac{\sqrt{3}}{2} (t+b))$, and so
            \[\|F(x) - F(y)\|^2 = \frac{1}{4}\left(a - b - s + t\right)^2 + \frac{3}{4}\left(a + s + t + b \right)^2.\]
            Assume without loss of generality that $a = 1$ and $b \leq 1$. If $s+t \leq 1$, then
            \[\frac{\|F(x) - F(y)\|^2}{d(x,y)^2} \leq M_7(b,s,t),\]
            where
            \[ M_7(b,s,t) = \frac{1}{4}\left(1 - b - s + t\right)^2 + \frac{3}{4}\left(1 + s + t + b \right)^2.\]
            Over the region $\{(b,s,t): 0 \leq b,s,t \leq 1, s+t \leq 1\}$, $M_7(b,s,t)$ is maximized by taking $s = 0$, $t = 1$, $b = 1$ or $s = 1$, $t = 0$, $b = 1$. By \Cref{lemm:M7}, this gives the value $1/4 + 3/4 \cdot 3^2 = 7$, which gives the desired bound. 

            If $s+t \geq 1$, then 
            \[\frac{\|F(x) - F(y)\|^2}{d(x,y)^2} \leq \frac{M_7(b,s,t)}{(s+t)^2}.\]
            Over the region $\{(b,s,t): 0 \leq b,s,t, 1 \leq s+t\}$, the function $M_7(b,s,t)/(s+t)^2$ is maximized by taking  $s = 0$, $t = 1$, $b = 1$ or $s = 1$, $t = 0$, $b = 1$, which by \Cref{lemm:M7b} again gives the desired bound. 

            Next, we show the lower bound. Assume first that $z \in I_3$. Then we can show similarly to in \Cref{lemm:distortion_case4} that $d(y,z) \leq a + t - s$. Actually, for the remainder of the proof it is slightly more convenient to use the weaker bound $d(y,z) \leq a+ t + s$. From this we conclude that $d(x,y) \leq 2a + s+t$. Assume without loss of generality that $a=1$. Then $d(x,y)^2/\|F(x) - F(y)\|^2$ is bounded by the function
            \[M_8(b,s,t) = \frac{(2+s+t)^2}{\frac{1}{4}\left(1 - b - s + t\right)^2 + \frac{3}{4}\left(1 + s + t + b \right)^2}, \]
            where $b,s,t \geq 0$. By \Cref{lemm:M8}, this function is bounded on the given region by $4$. 

            Next, the case that $w \in I_1$ is symmetric to the case that $z \in I_3$, and so we obtain the same bound.
            
            Finally, assume that both $z$ and $w$ are in $I_2$. We can show similarly to Case 2 in \Cref{lemm:distortion_case3} that $d(w,z) \leq a + b + s+t$. From this we conclude that $d(x,y) \leq 2a + 2b + s+t$. Assume without loss of generality that $a=1$. Then $d(x,y)^2/\|F(x) - F(y)\|^2$ is bounded by the function
            \[M_9(b,s,t) = \frac{(2+2b+2s+2t)^2}{\frac{1}{4}\left(1 - b - s + t\right)^2 + \frac{3}{4}\left(1 + s + t + b \right)^2},\]
            where $b,s,t \geq 0$. As in Case 3 of \Cref{lemm:distortion_case4}, we use a weaker bound in the numerator in order to simplify the calculations. By \Cref{lemm:M9}, this function is bounded above by $16/3$. 
        \end{proof}

    \section{Examples} \label{sec:examples}

    We give a detailed analysis of the examples introduced in \Cref{exm:examples}: the intrinsic circle, the three-petal rose, tripods and the twisted heart. Theses are denoted respectively by $S$, $R$, $T$ and $H$.

        \subsection{The intrinsic circle} \label{sec:circle}

    This example represents the ``fattest'' possible metric triangle. 
    The obvious embedding from $S$ into $\mathbb{R}^2$ is the identity map, which has distortion $\pi/2$. This embedding is indeed distortion-minimizing; this follows from a result of Linial--Magen  \cite[Claim 2.1]{LinMag:00} on embedding $n$-cycles into Euclidean space. See also Theorem 35 in \cite{CIM:24}.

        \subsection{The three-petal rose} \label{sec:rose}

    This example provides a lower bound on $\lip(\triangle)$, the minimal distortion required to embed every metric triangle in $\mathbb{R}^2$. 

    \begin{prop} \label{prop:rose_lower_bound}
        Let $G$ be an arbitrary embedding of $R$ into $\mathbb{R}^2$. Then $\lip(G) \geq 2$.  
    \end{prop}
    \begin{proof}
      Without loss of generality, we may assume that $G(o) = 0$ and that $G$ is non-expanding, but that any rescaling of $G$ by a factor $r>1$ is no longer non-expanding. In particular, for all $\varepsilon >0$, we can find points $x,y \in R$ such that $d(x,y)/\|G(x) - G(y)\| \leq 1 + \varepsilon$. 

      The triangle $R$ is the union of six arcs $S_i^1, S_i^2$, $i \in \{1,2,3\}$ with one endpoint $o$ and other endpoint the antipodal point of $S_i$, which we denote by $y_i$.  
      
      Choose a value $\varepsilon \in (0, 1/(2\lip(G)))$. Observe that $\|G(y_i)\| \geq \varepsilon$ for each $i$. In particular, since each set $G(S_i^j)$ is connected, we can find a point $x_i^j \in S_i^j$ such that $\|G(x_i^j)\|= \varepsilon$. 

      Enumerate the points $x_i^j$ based on the order of $G(x_i^j)$ along the circle $S(0,\varepsilon)$ (with the counterclockwise orientation) from the positive $x$-axis as $x_1,x_2, \ldots, x_6$. We can find consecutive points $x_i, x_{i+1}$ (with the identification $x_7= x_1$) such that the angle $\angle(0; G(x_i), G(x_{i+1})) \leq \pi/3$. In particular, $\|G(x_i) - G(x_{i+1})\| \leq \varepsilon$. On the other hand, since $G$ is non-expanding, $d(x_i,x_{i+1}) = d(x_i,o) + d(0,x_{i+1}) \geq 2\varepsilon$. 

      We conclude that \[\frac{d(x_i, x_{i+1})}{\|G(x_i) - G(x_{i+1})\|} \geq \frac{2\varepsilon}{\varepsilon} = 2,\] 
      and hence that the distortion of $G$ is at least $2\varepsilon/\varepsilon = 2$.
    \end{proof}

The tripodal embedding $F$ for the three-petal rose has distortion $\lip(F) = 2$. Alternatively, let $X$ be the union of three equilateral triangles with the same side length and sharing a common vertex, each separated by an angle of $\pi/3$. Then any arc length-preserving homeomorphism from $R$ onto $X$ also has distortion $2$. 

	\subsection{Tripods} \label{sec:tripod}

    As we have seen, tripods play a key role in the proof of \Cref{thm:embedding}. The class of tripods represents the ``thinnest'' possible metric triangles. We verify the natural expectation that the embedding $F$ from the previous section is distortion minimizing for this class of triangles. Note in this case that $d(x,\widehat{I}_j) = 0$ for all $x \in I_j$ ($j = 1,2,3$), and so $F(x) = \bar{x}$ for all points $x$. 

    \begin{prop}
        Let $T$ be a tripod that is non-degenerate, meaning that $p_i>0$ for all $i \in \{1,2,3\}$. Then $\lip(T,\mathbb{R}^2) = 2/\sqrt{3}$, and this is obtained by the tripodal embedding.  
    \end{prop}
    \begin{proof}
        We show that the tripodal embedding $F$ has distortion $2/\sqrt{3}$. 
        If $x,y \in T_i$ for the same $i \in \{1,2,3\}$, then the tripodal embedding satisfies $\|F(x) - F(y)\| = d(x,y)$. Next, assume that $x \in T_i$ and $y \in T_j$ for some $i \neq j$. Assume that $i=1$ and $j=2$. Let $s = d(x,o)$ and $t = d(y,o)$. Then 
        \[\frac{\|F(x) - F(y)\|}{d(x,y)} \leq \frac{\|F(x) - F(o)\|}{d(x,y)} + \frac{\|F(o) - F(y)\|}{d(x,y)} = \frac{s}{s+t} + \frac{t}{s+t}\ = 1,\]
        and
        \[\frac{d(x,y)}{\|F(x) - F(y)\|} = \frac{s+t}{\|(s+\frac{t}{2}, \frac{\sqrt{3}t}{2}) \|} = \frac{s+t}{\sqrt{(s+\frac{t}{2})^2 + \frac{3}{4}t^2 }}.\]
        For fixed $t$, the expression on the right-hand side attains a maximum when $s=t$. To check this, assume without loss of generality that $t=1$ and let \[M(s) = \frac{(s+1)^2}{(s+\frac{1}{2})^2 + \frac{3}{4}},\]
        defined for $s \geq 0$. We compute $M'(s) = (1 - s^2)/(1 + s + s^2)^2$. Thus we have a single critical point at $s=1$. We evaluate $M(1) = 4/3$. Also, $M(0) = 1$ and $\lim_{s \to \infty} M(s) = 1$, which establishes that $M(1)$ is a maximum. 

        The proof that no embedding can have smaller distortion is similar to the proof of \Cref{prop:rose_lower_bound} but with three arcs emanating from the center point $o$ rather than six. We omit the details.
    \end{proof}

    \subsection{The twisted heart} \label{sec:heart}

    This example shows the sharpness of \Cref{thm:embedding}. Observe first that $d(b,e) = d(c,a) = 1$ and $d(a,b) = 4$. Moreover, $\bar{a} = \bar{b} = (0,0)$ and $\bar{c} = (-\frac{1}{2}, \frac{\sqrt{3}}{2})$. We now compute that $F(a) = \bar{a} + v_1 = (\frac{1}{2}, \frac{\sqrt{3}}{2})$, $F(b) = \bar{b} + \bar{v}_3 = (\frac{1}{2}, -\frac{\sqrt{3}}{2})$ and $F(c) = \bar{c} + v_2 =  (-\frac{3}{2},-\frac{\sqrt{3}}{2})$. This yields the equalities
    \[\frac{\|F(a) - F(b)\|}{d(a,b)} = \frac{\|(0,\sqrt{3})\|}{4} = \frac{\sqrt{3}}{4}\]
    and 
    \[\frac{\|F(a) - F(c)\|}{d(a,c)} = \|(2,\sqrt{3})\| = \sqrt{7}.\]
    Thus the tripodal embedding on the twisted heart has distortion $4\sqrt{7/3}$.

    The natural follow-up is to what extent this value can be reduced by using a different embedding. One candidate is to map $H$ to an equilateral triangle in the obvious way; this gives a distortion of $2 \sqrt{7} \approx 5.29$. By perturbing the image points, it seems possible from computations to reduce the distortion to roughly $4$. We leave this example as an area for further exploration. 

        \begin{ques}
        What is $\lip(H,\mathbb{R}^2)$, and what embedding attains it?
    \end{ques}

 \section{Metric polygons} \label{sec:quadrilaterals}

 In this section, we prove \Cref{prop:bad_quadrilaterals}, which shows that the embedding result for metric triangles does not extend to the case of metric quadrilaterals.

 We define for all $\varepsilon>0$ a quadrilateral $Q = Q(\varepsilon)$ as follows. First we define a corresponding superset $\widetilde{Q}$. Start with a circle of length $4$ divided into four subarcs of length $1$, enumerated in cyclic order as $I_1, I_2, I_3, I_4$. For each $j = 1,\ldots, 4$, let $p_j$ denote the initial point of $I_j$ and $m_j$ denote the midpoint. Let $I_5$ be an interval of length $\varepsilon$ connecting the point $p_1$ to $p_3$, and let $I_6$ be an interval of length $\varepsilon$ connecting the point $p_2$ to $p_4$. Let $I_7$ be an interval of length $\varepsilon$ connecting $m_1$ and $m_3$. We now define $\widetilde{Q}$ to be the union of $I_1, \ldots, I_{7}$. We equip $\widetilde{Q}$ with the length metric, denoted by $d$. We then take $Q$ to be the subspace $I_1 \cup \cdots \cup I_4$. One checks that $Q$ is a metric quadrilateral with vertices $p_1,\ldots, p_4$.

 \Cref{prop:bad_quadrilaterals} is a consequence of the following proposition. 

 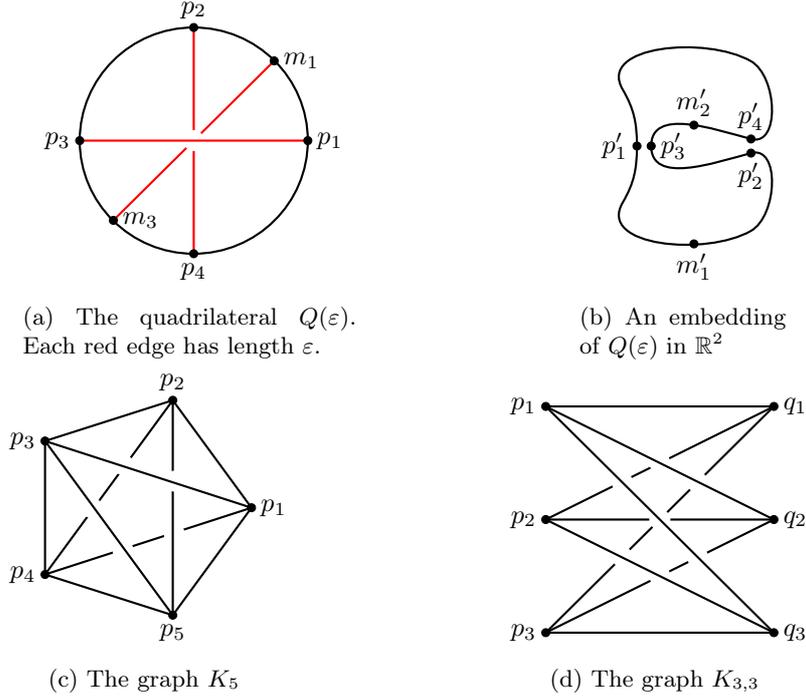
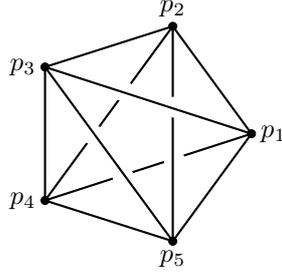
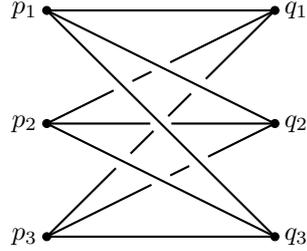
\begin{figure} 
    \centering
    \hfill
    \subfloat[The quadrilateral $Q(\varepsilon)$. Each red edge has length $\varepsilon$.]{
        \begin{tikzpicture}[scale=.75]
        \draw[black, thick] (0,0) circle (2);
        \draw[red,thick] (0,2) -- (0,-2);
        \draw[red,thick] (1.41,1.41) -- (-1.41,-1.41);
        \filldraw[white] (0,0) circle (5pt);
        \draw[red, thick] (2,0)--(-2,0);        
        \filldraw[black] (2,0) circle (2pt) node[anchor=west]{$p_1$};
        \filldraw[black] (0,2) circle (2pt) node[anchor=south]{$p_2$};
        \filldraw[black] (-2,0) circle (2pt) node[anchor=east]{$p_3$};
        \filldraw[black] (0,-2) circle (2pt) node[anchor=north]{$p_4$};
        \filldraw[black] (1.41,1.41) circle (2pt) node[anchor=west]{$m_1$};
        \filldraw[black] (-1.41,-1.41) circle (2pt) node[anchor=west]{$m_3$};
        \end{tikzpicture}
    \label{fig:Q}}
    \hfill
    \subfloat[An embedding of $Q(\varepsilon)$ in $\mathbb{R}^2$ and curve $S$ (in blue)]{
        \begin{tikzpicture}[scale=.75]
        \draw[black, thick] (-1.25,0) to[out=270,in=60] (-1.5,-1) to[out=240, in=210] (1,-1.75) to[out=30,in =15] (1,-.375) to[out=195,in=270] (-.75,0) to[out=90,in=165] (1,.375) to[out=-15,in=-30] (1,1.75) to[out=150, in=120] (-1.5,1) to[out=300,in=90] (-1.25,0);
        \draw[blue, very thick] (1,-.375) to (1,.375) to[out=-15,in=-30] (1,1.75) to[out=150, in=120] (-1.5,1) to[out=300,in=90] (-1.25,0);
        \draw[blue, very thick] (1,-.375) to[out=195,in=270] (-.75,0) to (-1.25,0);
        \filldraw[black] (-1.25,0) circle (2pt) node[anchor=east]{$j_1$};
        \filldraw[black] (1,-.375) circle (2pt) node[anchor=north]{$j_2$};
        \filldraw[black] (-.75,0) circle (2pt) node[anchor=west]{$j_3$};
        \filldraw[black] (1,.375) circle (2pt) node[anchor=south]{$j_4$};
        \filldraw[black] (0,-1.98) circle (2pt) node[anchor=north]{$m_1'$};
        \filldraw[black] (0,.53) circle (2pt) node[anchor=south]{$m_3'$};
        \end{tikzpicture}
    \label{fig:Qembed}}
    \hfill \hfill

    \hfill
    \subfloat[The graph $K_5$]{
        \begin{tikzpicture}[scale=.75]
        \filldraw[black] (2,0) circle (2pt) node[anchor=west]{$p_1$};
        \filldraw[black] (.62,1.9) circle (2pt) node[anchor=south]{$p_2$};
        \filldraw[black] (-1.62,1.18) circle (2pt) node[anchor=east]{$p_3$};
        \filldraw[black] (-1.62,-1.18) circle (2pt) node[anchor=east]{$p_4$};
        \filldraw[black] (.62,-1.9) circle (2pt) node[anchor=north]{$p_5$};
        \draw[black, thick] (2,0) to (.62,1.9) to (-1.62,1.18) to (-1.62,-1.18) to (.62, -1.9) to (2,0);
        \draw[black, thick] (-1.62, -1.18) to (2,0);
        \filldraw[white] (-.25,-.76) circle (5pt);
        \filldraw[white] (.65,-.47) circle (5pt);
        \draw[black, thick] (.62, -1.9) to (.62,1.9) to (-1.62,-1.18);
        \filldraw[white] (.65,.47) circle (5pt);
        \filldraw[white] (-.25,.76) circle (5pt);
        \filldraw[white] (-.8,0) circle (5pt);
        \draw[black, thick] (2,0) to (-1.62,1.18) to (.62, -1.9);
        \end{tikzpicture}
    \label{fig:K5}}
    \hfill
    \subfloat[The graph $K_{3,3}$]{
             \begin{tikzpicture}[scale=.75]
        \filldraw[black] (-2,2) circle (2pt) node[anchor=east]{$p_1$};
        \filldraw[black] (-2,0) circle (2pt) node[anchor=east]{$p_2$};
        \filldraw[black] (-2,-2) circle (2pt) node[anchor=east]{$p_3$};
        \filldraw[black] (2,2) circle (2pt) node[anchor=west]{$q_1$};
        \filldraw[black] (2,0) circle (2pt) node[anchor=west]{$q_2$};
        \filldraw[black] (2,-2) circle (2pt) node[anchor=west]{$q_3$};
        \draw[black, thick] (-2,-2)--(2,2);
        \draw[black, thick] (-2,-2)--(2,0);
        \draw[black, thick] (-2,-2)--(2,-2);
        \filldraw[white] (0,-1) circle (5pt);
        \filldraw[white] (.66,-.66) circle (5pt);
        \filldraw[white] (-.66,-.66) circle (5pt);
        \draw[black, thick] (-2,0)--(2,2);
        \draw[black, thick] (-2,0)--(2,0);
        \draw[black, thick] (-2,0)--(2,-2);
        \filldraw[white] (0,1) circle (5pt);
        \filldraw[white] (0,0) circle (5pt);
        \filldraw[white] (-.66,.66) circle (5pt);\filldraw[white] (.66,.66) circle (5pt);
        \draw[black, thick] (-2,2)--(2,2);
        \draw[black, thick] (-2,2)--(2,0);
        \draw[black, thick] (-2,2)--(2,-2);
        \end{tikzpicture}
    \label{fig:K33}}
    \hfill \hfill
    \caption{The metric polygons in \Cref{sec:quadrilaterals}.}
    \label{quadrilaterals}
\end{figure}

 \begin{prop}
     Fix $L \geq 1$ and choose $0 < \varepsilon < 1/(6L^4)$. Then the quadrilateral $Q = Q(\varepsilon)$ is not $L$-bi-Lipschitz embeddable in the plane. 
 \end{prop}
 \begin{proof}
    Assume that there is an $L$-bi-Lipschitz embedding $f \colon Q \to \mathbb{R}^2$. Given a point $x \in Q$, we write $x'$ in place of $f(x)$. 
    Observe that $f(Q)$ is a simple closed curve that, by the Jordan curve theorem, separates the plane and has two complementary components.

    Let $\widetilde{J}_1$ be the straight line segment connecting $p_1'$ to $p_3'$, which has length at most $L\varepsilon$. We can find a maximal closed subinterval $J_1$ with interior contained in $\mathbb{R}^2 \setminus f(Q)$ such that one endpoint $j_1 \in \widetilde{J}_1$ belongs to $f(I_1 \cup I_4)$ and the other endpoint $j_3 \in \widetilde{J}_1$ belongs to $f(I_2 \cup I_3)$. Observe that necessarily $\|p_1' - j_1\| < L\varepsilon$ and $\|p_3' - j_3\| < L\varepsilon$. We define $J_2$ and points $j_2$ and $j_4$ likewise, with $\|p_2' - j_2\| < L \varepsilon$ and $\|p_4' - j_4\| < L \varepsilon$. We further observe that $J_1$ and $J_2$ do not intersect; if there were a point $z \in J_1 \cup J_2$, then we would have $\|p_1' - p_2'\| \leq \|p_1' - z\| + \|z - p_2'\| \leq 2L\varepsilon$, which contradicts the statement that $\|p_1' - p_2'\| \geq 1/L$. Thus $J_1$ and $J_2$ lie in different complementary components of $f(Q)$. 

    We claim that the points $j_1$, $m_1'$, $j_2$, $j_3$, $m_3'$, $j_4$ are arranged cyclically along the simple closed curve $f(Q)$ in this order. To see this, let $H_i$ denote the shorter subarc of $f(Q)$ from $p_i'$ to $j_i$ for each $1 \leq i \leq 4$. Each $H_i$ has length at most $L^3\varepsilon$, since
    \[\ell(H_i) \leq L \ell(f^{-1}(H_i)) = Ld(p_1, f^{-1}(j_1)) \leq L^2\|p_1' - j_1\| \leq L^3\varepsilon.\]
    On the other hand, $\|m_1' - p_1'\| \geq 1/(2L) > L^3 \varepsilon$, which implies that $m_1'$ cannot lie on the subarc $H_1$. Similar relations hold for other pairs of points. We conclude that the points $j_1$, $m_1'$, $j_2$, $j_3$, $m_3'$, $j_4$ are ordered the same as $p_1'$, $m_1'$, $p_2'$, $p_3'$, $m_3'$, $p_4'$, which establishes the claim.

    Now let $K_1$ be the subarc of $f(Q)$ from $j_2$ to $j_3$ and $K_2$ be the subarc of $f(Q)$ from $j_4$ to $j_1$ given by the previous ordering, and let $S$ be the concatenation of $K_1$, $J_1$, $K_2$ and $J_2$. Then $S$ is a simple closed curve for which $m_1'$ and $m_3'$ lie in different complementary components. See \Cref{fig:Qembed}. Now $\|m_1' - m_3'\| \leq L\varepsilon$, which implies that $m_1'$ and $m_3'$ lie at distance at most $L\varepsilon < (6L^3)^{-1} \leq (6L)^{-1}$ from $S$.

    On the other hand, the distance from $m_1'$ to $K_1$ is at least 
    \[\text{dist}_{\text{Euc}}(m_1', f(I_2)) - \max\{\ell(H_2), \ell(H_3)\} \geq \frac{1}{2L} - L^3\varepsilon \geq \frac{1}{2L} - \frac{1}{6L} \geq \frac{1}{3L}.\] Likewise, the distance from $m_1'$ to $J_1$ is at least 
    \[\|m_1' -p_1'\| - \ell(\widetilde{J}_1) \geq \frac{1}{2L} - L\varepsilon \geq \frac{1}{3L}.\] 
    The same bound holds for $K_2$ and $J_2$, and so $d_{\text{Euc}}(m_1',S) \geq (3L)^{-1}$. This gives a contradiction, and the proposition follows.  
 \end{proof}

\begin{rem}
    One might initially assume that taking $\varepsilon = 0$ in the previous construction gives a metric quadrilateral that cannot be bi-Lipschitz embedded in $\mathbb{R}^2$. While indeed we do get a metric quadrilateral in this way, it is possible to bi-Lipschitz embed it in $\mathbb{R}^2$. The previous proof relies on the fact that a simple closed curve separates the plane. However, if $\varepsilon=0$, the corresponding quadrilateral is not a simple closed curve, and there is more flexibility for where to map each edge.
\end{rem}

 \begin{exm} \label{exm:pentagon}
    A metric pentagon may fail to even be topologically embeddable in $\mathbb{R}^2$. There are two basic examples of non-planar graphs: $K_5$, the complete graph on five vertices, and $K_{3,3}$, the complete bipartite graph on two sets of three vertices. Each of these can be realized as a metric pentagon.

    We provide the details for $K_5$. Denote its topological vertices in cyclic order by $p_1, p_2, \ldots, p_5$. Denote the topological edge from $p_i$ to $p_j$ by $I_{ij}$. We make $K_5$ into a metric space by assigning each edge length $1$ and taking the length metric. For each $i = 1,\ldots, 5$, let $m_i$ be the midpoint of $I_{i(i+1)}$ (with the usual identification of $I_{56}$ with $I_{51}$). The space $K_5$ can be viewed as a metric pentagon by taking $m_1, m_3, m_5, m_2, m_4$ as its vertices (as a metric polygon) in cyclic order. The edges (as a metric polygon) are $[m_1p_1] * I_{14} * [p_4m_3]$, $[m_3p_3] * I_{31} * [p_1m_5]$, $[m_5p_5] * I_{53} * [p_3m_2]$, $[m_2p_2] * I_{25} * [p_5m_4]$, $[m_4p_4] * I_{42} * [p_2m_1]$. Each of these edges has length $2$.

    Note that, for both metric spaces $K_5$ and $K_{3,3}$ (with the standard length metric just described), any two points are at distance at most $2$ apart. This implies that any edge (as a metric polygon) has length at most $2$. In particular, these spaces cannot be realized as metric quadrilaterals. 
 \end{exm}

 \section{Appendix} \label{sec:appendix}

 In this appendix, we provide the proofs of the various inequalities used in \Cref{sec:embedding}. These are mostly straightforward calculus optimization problems, but we include the details for completeness. We recall the following functions defined in \Cref{sec:embedding}:
\begin{align*}
    & M_1(b,e) = \frac{1}{4}\left(1 - b + 2e \right)^2 + \frac{3}{4} \left(1 + b\right)^2;  \\
    & M_2(b,e) = \frac{(2 + e)^2}{\frac{1}{4}\left(1 - b + 2e\right)^2 + \frac{3}{4}(1+b)^2};  \\
    & \widetilde{M}_2(b,e) = \frac{(2 - e)^2}{\frac{1}{4}\left(1 - b + 2e\right)^2 + \frac{3}{4}(1+b)^2}; \\
    & M_3(b,s,t) = \frac{(2+2b-s-t)^2}{\frac{1}{4}(1 - b + 2s -2t)^2 + \frac{3}{4}(1 + b)^2}; \\
    & M_4(b,s,t) = \frac{1}{4}\left(1 - b + 2s + t\right)^2 + \frac{3}{4}\left(1 + t + b \right)^2; \\
    & M_5(b,s,t) = \frac{(\max\{2,2b\} + s + t)^2}{\frac{1}{4}\left(1 - b + 2s  + t\right)^2 + \frac{3}{4}\left(1 + t + b\right)^2}; \\
    & M_6(b,s,t) = \frac{(2+2b+2t)^2}{\frac{1}{4}\left(1 - b + 2s  + t\right)^2 + \frac{3}{4}\left(1 + t + b\right)^2}; \\
    & M_7(b,s,t) = \frac{1}{4}\left(1 - b - s + t\right)^2 + \frac{3}{4}\left(1 + s + t + b \right)^2; \\
    & M_8(b,s,t) = \frac{(2+s+t)^2}{\frac{1}{4}\left(1 - b - s + t\right)^2 + \frac{3}{4}\left(1 + s + t + b \right)^2}; \\
    & M_9(b,s,t) = \frac{(2+2b+2s+2t)^2}{\frac{1}{4}\left(1 - b - s + t\right)^2 + \frac{3}{4}\left(1 + s + t + b \right)^2}.
\end{align*}

 \begin{lemm} \label{lemm:M1}
    The function $M_1(b,e)$ is maximized on the set $\{(b,e): 0 \leq b,e \leq 1\}$ by taking $b=e=1$, which gives a value of $4$.
 \end{lemm}
 \begin{proof}
     We compute $\frac{\partial M_1}{\partial b}(b,e) = 1+2b-e$ and $\frac{\partial M_1}{\partial e}(b,e) = 1-b + 2e$. These are both non-negative on the set $\{(b,e): 0 \leq b,e \leq 1\}$. Thus $M_1$ attains a maximum value by making $b,e$ as large as possible, that is, $b=e=1$. We compute $M_1(1,1)=4$.
 \end{proof}

  \begin{lemm} \label{lemm:M1b}
    The function $\frac{1}{e^2}M_1(b,e)$ is maximized on the set $\{(b,e): 0 \leq b \leq 1, 1 \leq e <\infty\}$ by taking $b=e=1$, which gives a value of $4$.
 \end{lemm}
 \begin{proof}
    We first compute
    \[\frac{\partial }{\partial e}\frac{M_1(b,e)}{e^2} = -\frac{2 + 2 b + 2 b^2 + e - b e}{e^3}.\]
    Observe that since $b \leq 1$, the numerator is always positive, and we conclude that $\frac{\partial M_1}{\partial e}(b,e) < 0$. This implies that $\frac{1}{e^2}M_1(b,e)$ attains a maximum when $e=1$. The previous lemma now implies that a maximum value of $4$ is attained when $b=1$.
 \end{proof}

 \begin{lemm} \label{lemm:M2}
    The function $M_2(b,e)$ is maximized on the set $\{(b,e): b \geq 0\}$ by taking $b= e = 0$, which gives a value of $4$.     
 \end{lemm}
 \begin{proof}
    We compute 
     \[\frac{\partial M_2}{\partial b}(b,e) = -\frac{(2 + e)^2 (1 + 2 b - e) }{(\frac{1}{4}\left(1 - b + 2e\right)^2 + \frac{3}{4}(1+b)^2)^2}.\]
    Observe that $\frac{\partial M_2}{\partial b}(b,e)>0$ if $2b<e-1$ and $\frac{\partial M_2}{\partial b}(b,e)<0$ if $2b > e-1$. This implies that $M_2(b,e)$ is greatest if $2b + 1 = e$, or if $b=0$. 

    Now 
    \[\frac{\partial}{\partial b} M_2(b,2b+1) = -\frac{2(3+2b)}{3 (1 + b)^3},\]
    which is negative. Thus $M_2(b,2b+1)$ is maximized when $b=0$, which gives a value of $M(0,1) = 4/3$. 

    Similarly, if $b=0$, then we compute
    \[\frac{\partial}{\partial e} M_2(0,e) = -\frac{3 e (2 + e)}{(1 + e + e^2)^2},\]
    which is negative if $e>0$ or $e<-2$ and positive if $-2 < e < 0$. Thus $M_2(0,e)$ has a local maximum at $e=0$, which gives the value $M_2(0,0)=4$. We also check that $\lim_{e \to - \infty} M_2(0,e) = 1$, and so $e=0$ is a global maximum. This completes the proof.     
 \end{proof}

 \begin{lemm} \label{lemm:M2b}
     The function $\widetilde{M}_2(b,e)$ is maximized on the set $\{(b,e): b \geq 0, -b \leq e \leq 0\}$ by taking $b= -e = 1/6$, which gives a value of $13/3$.
 \end{lemm}
 \begin{proof}
     We compute 
     \[\frac{\partial \widetilde{M}_2}{\partial b}(b,e) = -\frac{(1 + 2 b - e) (e-2)^2}{(\frac{1}{4}\left(1 - b + 2e\right)^2 + \frac{3}{4}(1+b)^2)^2}.\]
     This is strictly negative over the given region. It follows that $\widetilde{M}_2(b,e)$ is maximized if $b=-e$. Now
     \[\frac{\partial}{\partial e}\widetilde{M}_2(-e,e)  = \frac{2 (e-2) (1 + 6e)}{(1 + 3 e^2)^2}.\]
    This is positive if $e<-1/6$ and negative if $e > -1/6$. Thus there is a maximum at $e = -1/6$. We evaluate 
    $\widetilde{M}_2(1/6,-1/6) = 13/3$.    
 \end{proof}

 \begin{lemm} \label{lemm:M3}
     The function $M_3(b,s,t)$ is maximized on the set $\{(b,s,t): 0 \leq b,s,t \leq 1\}$ by taking $b=1$, $s=t=0$, which gives a value of $16/3$. 
 \end{lemm}
 \begin{proof}
         We compute
     \[\frac{\partial M_3}{\partial s}(b,s,t) = -\frac{(2 + 2 b - s - t)(s (5 - 4 t) + b (2 + 3 s - t) + 4 (t-\frac{7}{8})^2 + \frac{15}{16})}{(\frac{1}{4}(1 - b + 2 s - 2 t)^2 + \frac{3}{4} (1 + b)^2)^2},\]
     which is negative for all $0 < b,s,t < 1$.   It follows that $M_3(b,s,t)$ is maximized if $s=0$. Next, we evaluate
    \[\frac{\partial}{\partial t}M_3(b,0,t) = -\frac{(2 + 2 b - t) (4b^2 + b(5t+2) + 3t)}{2(\frac{1}{4}(1 - b - 2t)^2 + \frac{3}{4} (1 + b)^2)^2}, \]
   which is also negative for all $0 < b,t < 1$. Thus $M_3(b,0,t)$ is maximized if $t=0$. Finally, we evaluate 
   \[\frac{\partial M_3}{\partial b}(b,0,0) = \frac{4(1-b^2)}{(\frac{1}{4}(1 - b)^2 + \frac{3}{4} (1 + b)^2)^2},\]
   which is positive if $0 < b< 1$. We conclude that $M_3(b,s,t)$ is maximized when $b=1$, $s=t=0$. Evaluate $M_3(1,0,0) = 16/3$. 
 \end{proof}

 \begin{lemm} \label{lemm:M4}
     The function $M_4(b,s,t)$ is maximized on the set $\{(b,s,t): 0 \leq b,s,t \leq 1, s+t \leq 1\}$ by taking $s = 0$, $b= t = 1$, which gives a value of $7$. 
 \end{lemm}
 \begin{proof}
     Since $1-b+2s+t \geq 0$, we see that $M_4(b,s,t)$ is increasing in $s$ and $t$. Thus the maximum is attained when $s+t = 1$. Now 
     \[M_4(b,s,1-s) = \frac{3}{4} (2 + b - s)^2 + \frac{1}{4} (2 - b + s)^2.\]
     Let $u = b-s$, so the previous expression becomes $h(u) = \frac{3}{4}(2+u)^2 + \frac{1}{4}(2-u)^2$. Then $h'(u) = 2u+2$. So $M_4$ is maximized when $b-s$ is greatest, that is, $b=1$ and $s=0$. Then $t=1$. We evaluate $M_4(1,0,1) = 7$. 
 \end{proof}

 \begin{lemm} \label{lemm:M4b}
    The function $\frac{1}{(s+t)^2}M_4(b,s,t)$ is maximized on the set
     $\{(b,s,t): 0 \leq b, s,t, 1 \leq s+t, b \leq s+t\}$ by taking $s = 0$, $b= t = 1$, which gives a value of $7$.
 \end{lemm}
 \begin{proof}
     We evaluate
     \[\frac{\partial}{\partial b} \frac{M_4(b,s,t)}{(s+t)^2} = \frac{1 + 2 b - s + t}{(s + t)^2},\]
     which is positive when $2b > s-t-1$. If $2b > s-t-1$, then $\frac{1}{(s+t)^2}M_4(b,s,t)$ increases as $b$ increases. If $2b < s - t-1$, then $\frac{1}{(s+t)^2}M_4(b,s,t)$ increases as $b$ decreases.    
     From this we conclude that $\frac{1}{(s+t)^2}M_4(b,s,t)$ is greatest when either $b=s+t$ or $b= 0$. When analyzing the latter case, note that $t+1 < s$. 
     Then 
     \[M_4(s+t,s,t) = \frac{s^2 + 2s + 3ts + 3 t^2 + 3 t + 1}{(s + t)^2}\]
     and we compute
     \[\frac{\partial}{\partial s} \frac{M_4(s+t,s,t)}{(s+t)^2} = -\frac{2 + 2s + st + 4 t + 3 t^2}{(s + t)^3},\]
     which is negative over the same region. Thus the function is greatest when $s+t = 1$. The previous lemma now implies that a maximum value of $7$ is attained when $s = 0$ and $b= t = 1$. 

     On the other hand,
      \[M_4(0,s,t) = \frac{s^2 + 2s + 3ts + 3 t^2 + 3 t + 1}{(s + t)^2}.\]
      Let $u = (0,\frac{1}{\sqrt{2}},-\frac{1}{\sqrt{2}})$.
     We compute
     \[D_u \frac{M_4(0,s,t)}{(s+t)^2} = \frac{s-t-1}{\sqrt{2}(s + t)^2},\]
     which is positive. From this, we can assume that $t=0$. Then $M_4(0,s,0) = 1 + \frac{1}{s} + \frac{1}{s^2}$, which is at most $3$.  
 \end{proof}

  \begin{lemm} \label{lemm:M4c}
    The function $\frac{1}{b^2}M_4(b,s,t)$ is maximized on the set
     $\{(b,s,t): 0 \leq b,s,t, 1 \leq b, s+t \leq b\}$ by taking $s = 0$, $b= t = 1$, which gives a value of $7$.
 \end{lemm} 
 \begin{proof}
     Let $u = (0,\frac{1}{\sqrt{2}}, \frac{1}{\sqrt{2}})$. The directional derivative of $\frac{1}{b^2}M_4(b,s,t)$ in the direction of $u$ is $3 (1 + s + t)/(\sqrt{2}b^2)$, which is positive. We conclude that $\frac{1}{b^2}M_4(b,s,t)$ is greatest when $s+t = b$. This reduces to \Cref{lemm:M4b}, and so we obtain the same conclusion. 
 \end{proof}

 \begin{lemm} \label{lemm:M5}
     The function $M_5(b,s,t)$ is maximized on the set $\{(b,s,t): 0 \leq b,s,t,\, s \leq 1\}$ by taking $b=6$, $s=1$, $t=0$, which gives a value of $13/3$. 
 \end{lemm}
 \begin{proof}
    Let $M_5^a(b,s,t)$ denote the function defined like $M_5$ but replacing ``$\min\{2,2b\}$'' with ``$2$'', and similarly for $M_5^b$. It suffices to consider the functions $M_5^a$ and $M_5^b$ separately.

     Let $u = (0, \frac{1}{\sqrt{2}}, -\frac{1}{\sqrt{2}})$. Then the directional derivative of $M_5^a$ in the $u$-direction is
     \[D_u M_5^a(b,s,t) = \frac{(1 + 2 b - s + t)(2+s+t)^2}{\sqrt{2}(\frac{1}{4} (1 - b + 2 s + t)^2 + \frac{3}{4} (1 + b + t)^2)^2}.\]
     If $s <1+2b+t$, then $D_u M_5^a$ is positive; if $s > 1+2b+t$, then $D_u M_5^a$ is negative. Thus along a given curve in the $u$-direction, $M_5^a(b,s,t)$ is greatest when $s = 1+2b+t$. Note that this point lies outside the interior of the target region $\{(b,s,t): 0 \leq b,s,t,\, s \leq 1\}$. From this we deduce that the maximum necessarily occurs when $t=0$ or when $s=1$.  

     In the case that $t=0$, we evaluate
     \[M_5^a(b,s,0)= \frac{(2+s)^2}{\frac{1}{4} (1 - b + 2 s)^2 + \frac{3}{4} (1 + b)^2}.\]
     Compute
     \[\frac{\partial}{\partial b} \frac{(2+s)^2}{\frac{1}{4} (1 - b + 2 s)^2 + \frac{3}{4} (1 + b)^2} = -\frac{(1 + 2 b - s)(2 + s)^2}{(\frac{1}{4} (1 - b + 2 s)^2 + \frac{3}{4} (1 + b)^2)^2}.\]
     Since $s \leq 1$, this is negative, so we conclude that the maximum occurs when $b=0$. Now
     \[M_5^a(0,s,0) = \frac{(2+s)^2}{\frac{1}{4} (1 + 2 s)^2+ \frac{3}{4} }.\]
     This function is decreasing in $s$ for $s>0$, so we conclude that $s=0$. We compute $M_5^a(0,0,0)=4$, so $M_5^a$ attains a maximum value for $b=s=t=0$.

     In the case that $s=1$, we evaluate
     \[M_5^a(b,1,t) = \frac{(3 + t)^2}{\frac{1}{4} (3 - b + t)^2 + \frac{3}{4} (1 + b + t)^2}.\]
     A computation shows that the $b$-derivative of the right-hand side is negative, and thus any maximum occurs when $b=0$. We then deduce that $t=0$. Evaluate $M_5^a(0,1,0) = 3$.  

     The function $M_5^b$ is handled similarly and we highlight the differences. We compute
    \[D_u M_5^b(b,s,t) = \frac{(1 + 2 b - s + t)(2b+s+t)^2}{\sqrt{2}(\frac{1}{4} (1 - b + 2 s + t)^2 + \frac{3}{4} (1 + b + t)^2)^2}.\]
    From this, we conclude that the maximum occurs when $s = 1 + 2b + t$, when $t=0$ or when $s=1$. As before, the line $s=1+2b+t$ lies outside the target domain, so we assume that $t=0$ or $s=1$. 

    If $t=0$, we compute
    \[\frac{\partial}{\partial s} M_5^b(b,s,0) =  \frac{(2 b + s) (2 + 4 b^2 + s - 5 b s)}{(\frac{1}{4} (1 - b + 2 s + t)^2 + \frac{3}{4} (1 + b + t)^2)^2}.\]
    This is positive for $b \geq 0$, $0 < s< 1$, so we conclude that $s=1$. 

    Thus we can consider the second case that $s=1$. We compute
    \[\frac{\partial}{\partial b} M_5^b(b,1,t) = - \frac{(1 + 2 b + t) (-12 + 2 b - 11 t - 3 t^2)}{(\frac{1}{4} (1 - b + 2 s + t)^2 + \frac{3}{4} (1 + b + t)^2)^2}\]
    and
    \[\frac{\partial}{\partial t} M_5^b(b,1,t) = - \frac{(1 + 2 b + t) (-3 +  7b - t + 3 bt)}{(\frac{1}{4} (1 - b + 2 s + t)^2 + \frac{3}{4} (1 + b + t)^2)^2}.\]
    The first of these is zero if $b = -(11t+3t^2)/12$, while the second of these is zero if $b = (3+t)/(7+3t)$. There is no $t>0$ satisfying both these relations, and so we conclude that $M_5^b(b,1,t)$ does not have any critical points when $0 < b,t$. Thus any maximum must occur when $t=0$ or $b=0$. 

    Next,
    \[\frac{\partial}{\partial b} M_5^b(b,1,0) =  -\frac{2 (-6 + b) (1 + 2 b)}{(3+b^2)^2},\]
    so $M_5^b(b,1,0)$ has a maximum at $b=6$. Compute $M_5^b(6,1,0) = 13/3$. Similarly, 
    \[\frac{\partial}{\partial b} M_5^b(0,1,t) =  -\frac{(1 + t) (3 + t)}{(3 + 3 t + t^2)^2},\]
    which is positive for $t>0$. So
    $\lim M_5^b(0,1,t) = 1$ is a bound on $M_5^b(0,1,t)$ for $t>0$. 
    
    Finally, $\frac{\partial}{\partial t} M_5^b(b,1,t)$ is negative whenever $b \geq 1$, which implies $M_5^b(b,1,t) \leq M_5^b(b,1,0)$ for all $b,t$ satisfies $b \geq 1$. Also, $\frac{\partial}{\partial b} M_5^b(b,1,t)$ is positive whenever $b \leq 1$. These facts allow us to conclude that $M_5^b(b,1,t)$ does in fact attain a maximum. 
 \end{proof}

 \begin{lemm} \label{lemm:M6}
     The function $M_6(b,s,t)$ is maximized on the set $\{(b,s,t): 0 \leq b,s,t\}$ by taking $s+t = 1+b$, which gives a value of $16/3$. 
 \end{lemm}
 \begin{proof}
     Compute
     \[\frac{\partial M_6}{\partial s} (b,s,t) = \frac{4 (-1 + b - 2 s - t) (1 + b + t)^2}{(\frac{1}{4} (1 - b + 2 s + t)^2 + \frac{3}{4} (1 + t + b)^2)^2}.\]
     This is positive if $2s <-1+b-t$ and negative if $2s> -1+b-t$. Thus any maximum of $M_6$ occurs when $2s = -1+b-t$, if $-1+b-t \geq 0$, or $-1+b-t<0$ and $s=0$. In the first case, we evaluate $M_6(b,(-1+b-t)/2,t) = 16/3$. 

    In the second case, evaluate
    \[\frac{\partial M_6}{\partial b} (b,0,t) = -\frac{4 (-1 + b - t) (1 + t) (1 + b + t)}{(b^2 + b (1 + t) + (1 + t)^2)^2}. \]
    If $-1+b-t<0$, then this derivative is positive. We conclude that any maximum for $M_6(b,0,t)$ occurs when $-1+b-t=0$, which returns us to the previous case. 
 \end{proof}

 \begin{lemm} \label{lemm:M7}
     The function $M_7(b,s,t)$ is maximized on the set $\{(b,s,t): 0 \leq b,s,t \leq 1, s+t \leq 1\}$, $M_7(b,s,t)$ by taking $s = 0$, $b=t = 1$, or $t=0$, $b = s = 1$. This gives a value of $7$. 
 \end{lemm}
 \begin{proof}
 We evaluate 
 \[\frac{\partial}{\partial b} M_7(b,s,t) = \frac{\partial}{\partial s} M_7(b,s,t) =  1 + 2 b + 2 s + t,\]
 which is positive. Thus $M_7$ is greatest when $b=1$ and $s+t = 1$. Now
 \[M_7(1,s,1-s) = \frac{27}{4} + \frac{1}{4} (1 - 2 s)^2,\]
 which attains a maximum of $7$ when $s=0$ or $s=1$. 
 \end{proof}

 \begin{lemm} \label{lemm:M7b}
    The function $\frac{1}{(s+t)^2}M_7(b,s,t)$ is maximized on the set $\{(b,s,t): 0 \leq b,s,t \leq 1, s+t \geq 1\}$, $M_7(b,s,t)$ by taking $s = 0$, $b=t = 1$, or $t = 0$, $b = s = 1$. This gives a value of $7$.
 \end{lemm}
 \begin{proof}
      We evaluate
     \[\frac{\partial}{\partial b} \frac{M_7(b,s,t)}{(s+t)^2} = \frac{1 + 2 b + 2 s + t}{(s + t)^2},\]
     which is positive for $0 < b,s,t<1$. This implies the function is greatest when $b=1$. We now compute
     \[\frac{\partial}{\partial s} \frac{M_7(1,s,t)}{(s+t)^2} = \frac{-6 + ts-3s - 3 t - t^2}{(s + t)^3},\]
     which is negative over the same region. Thus the function is greatest when $s+t = 1$. The previous lemma now implies that a maximum value of $7$ is attained when $s = 0$ and $b= t = 1$, or $t=0$ and $b = s = 1$.
 \end{proof}

\begin{lemm} \label{lemm:M8}
    The function $M_8(b,s,t)$ is maximized on the set $\{(b,s,t): 0 \leq b,s,t\}$ by taking $b=s=t=0$. This gives a value of $4$. 
\end{lemm}
\begin{proof}
     Let $u = (0, \frac{1}{\sqrt{2}}, -\frac{1}{\sqrt{2}})$. Then the directional derivative of $M_8$ in the $u$-direction is 
     \[D_u M_8(b,s,t) = \frac{(b-1+s-t)(2+s+t)^2}{\sqrt{2}(\frac{1}{4} (1 - b - s + t)^2 + \frac{3}{4} (1 + b + s + t)^2)^2}.\]
     If $s-t < 1-b$, then $D_u M_5$ is positive; if $s-t > 1-b$, then $D_u M_5$ is negative. Thus along a given curve in the $u$-direction, $M_8(b,s,t)$ is maximized when $s = 1 - b + t$. Note that this point may lie outside the target region $\{(b,s,t): 0 \leq b,s,t\}$, in which case the maximum necessarily occurs when $t=0$. We examine each case separately. First, we take $s = 1 - b + t$. Then
     \[M_8(b, 1-b+t,t) = \frac{(-3 + b - 2 t)^2}{3 (1 + t)^2}. \]
     This function is decreasing in $b$ if $b< 3+2t$ and increasing in $b$ if $b > 3+2t$. Moreover, for fixed $t$, since $s\geq 0$, it follows that $b$ satisfies the relation $b \leq 1+t$. Thus the maximum must occur when $b=0$ or $b= 1+t$. This gives
     \[M_8(0,1+t,t) = \frac{4 (3 + 2 t)^2}{3 (2 + 2 t)^2}\]
     and 
     \[M_8(1+t,0,t) = \frac{4 (2 + t)^2}{3 (2 + 2 t)^2} .\]
     Each of these functions is decreasing. So each has a maximum when $t=0$. We evaluate $M_8(0,1,0) = 3$ and $M_8(1,0,0)= 4/3$. 

     Next, we take $t=0$. Then
     \[M_8(b,s,0)= \frac{(2+s)^2}{\frac{1}{4} (1 - b - s)^2 + \frac{3}{4} (1 + b + s)^2}.\]
     Compute
     \[\frac{\partial}{\partial b} \frac{(2+s)^2}{\frac{1}{4} (1 - b - s)^2 + \frac{3}{4} (1 + b + s)^2} = -\frac{((2 + s)^2 (1 + 2 b + 2 s)}{1 + b + b^2 + s + 2 b s + s^2)^2}.\]
     This is negative, so we conclude that the maximum occurs when $b=0$. Now
     \[M_8(0,s,0) = \frac{(2+s)^2}{\frac{1}{4} (1 - s)^2 + \frac{3}{4} (1 + s)^2}.\]
     This function is decreasing in $s$ for $s>0$, so we conclude that $s=0$. We compute $M_8(0,0,0)=4$, so $M_8$ attains a maximum value for $b=s=t=0$. 
\end{proof}

\begin{lemm} \label{lemm:M9}
    The function $M_9(b,s,t)$ is maximized on the set $\{(b,s,t): 0 \leq b,s,t\}$ by taking $t = b+s-1$, which gives a value of $16/3$. 
\end{lemm}
\begin{proof}

    We compute
    \[\frac{\partial}{\partial t} M_9(b,s,t) = \frac{4 (b + s) (-1 + b + s - t) (1 + b + s + t)}{(\frac{1}{4} (1 - b - s + t)^2 + \frac{3}{4} (1 + b + s + t)^2)^2}.\]
    This is positive if $t + 1 < b+s$ and negative if $t+1 > b+s$, so we conclude that $M_9(b,s,t)$ is largest when $t = b+s - 1$. Now $M_9(b,s,b+s-1)$ evaluates to $16/3$.
\end{proof}

    \noindent \textbf{Acknowledgments:} We thank Guy C. David and Dustin Mixon for their feedback and corrections on a draft of this paper. We also thank the anonymous referees for their comments and corrections.

    \text{ }

    \noindent \textbf{Funding information:}    M.R. was partially supported by the National Science Foundation under grant DMS-2413156.

    \bibliographystyle{abbrv}
\bibliography{bibliography}

\begin{thebibliography}{10}

\bibitem{BH:99}
M.~R. Bridson and A.~Haefliger.
\newblock {\em Metric spaces of non-positive curvature}, volume 319 of {\em Grundlehren der mathematischen Wissenschaften [Fundamental Principles of Mathematical Sciences]}.
\newblock Springer-Verlag, Berlin, 1999.

\bibitem{BBI:01}
D.~Burago, Y.~Burago, and S.~Ivanov.
\newblock {\em A course in metric geometry}, volume~33 of {\em Graduate Studies in Mathematics}.
\newblock American Mathematical Society, Providence, RI, 2001.

\bibitem{CIM:24}
J.~Cahill, J.~W. Iverson, and D.~G. Mixon.
\newblock Towards a bilipschitz invariant theory.
\newblock {\em Appl. Comput. Harmon. Anal.}, 72:Paper No. 101669, 27, 2024.

\bibitem{DEV:23}
G.~C. David, S.~Eriksson-Bique, and V.~Vellis.
\newblock Bi-{L}ipschitz embeddings of quasiconformal trees.
\newblock {\em Proc. Amer. Math. Soc.}, 151(5):2031--2044, 2023.

\bibitem{DV:22}
G.~C. David and V.~Vellis.
\newblock Bi-{L}ipschitz geometry of quasiconformal trees.
\newblock {\em Illinois J. Math.}, 66(2):189--244, 2022.

\bibitem{Hei:03}
J.~Heinonen.
\newblock {\em Geometric embeddings of metric spaces}, volume~90 of {\em Report. University of Jyv\"{a}skyl\"{a} Department of Mathematics and Statistics}.
\newblock University of Jyv\"{a}skyl\"{a}, Jyv\"{a}skyl\"{a}, 2003.

\bibitem{LP:01}
U.~Lang and C.~Plaut.
\newblock Bilipschitz embeddings of metric spaces into space forms.
\newblock {\em Geom. Dedicata}, 87(1-3):285--307, 2001.

\bibitem{LLR:95}
N.~Linial, E.~London, and Y.~Rabinovich.
\newblock The geometry of graphs and some of its algorithmic applications.
\newblock {\em Combinatorica}, 15(2):215--245, 1995.

\bibitem{LinMag:00}
N.~Linial and A.~Magen.
\newblock Least-distortion {E}uclidean embeddings of graphs: products of cycles and expanders.
\newblock {\em J. Combin. Theory Ser. B}, 79(2):157--171, 2000.

\bibitem{Mat:02}
J.~Matou\v{s}ek.
\newblock {\em Lectures on discrete geometry}, volume 212 of {\em Graduate Texts in Mathematics}.
\newblock Springer-Verlag, New York, 2002.

\bibitem{NN:12}
A.~Naor and O.~Neiman.
\newblock Assouad's theorem with dimension independent of the snowflaking.
\newblock {\em Rev. Mat. Iberoam.}, 28(4):1123--1142, 2012.

\bibitem{NR:23}
D.~Ntalampekos and M.~Romney.
\newblock Polyhedral approximation of metric surfaces and applications to uniformization.
\newblock {\em Duke Math. J.}, 172(9):1673--1734, 2023.

\bibitem{Ost:13}
M.~I. Ostrovskii.
\newblock {\em Metric embeddings}, volume~49 of {\em De Gruyter Studies in Mathematics}.
\newblock De Gruyter, Berlin, 2013.
\newblock Bilipschitz and coarse embeddings into Banach spaces.

\bibitem{Sem:99}
S.~Semmes.
\newblock Bilipschitz embeddings of metric spaces into {E}uclidean spaces.
\newblock {\em Publ. Mat.}, 43(2):571--653, 1999.

\end{thebibliography}




\end{document}